\documentclass[11pt,twoside]{amsart}

\usepackage{enumerate}
\usepackage[all]{xy}
\usepackage[hyperindex=true,plainpages=false,colorlinks=false,pdfpagelabels]{hyperref}
\usepackage{verbatim}
\usepackage{graphicx,color}
\usepackage{subfigure}
\theoremstyle{plain}
\newtheorem{The}{Theorem}
\newtheorem*{The*}{Theorem}
\newtheorem{Pro}{Proposition}[section]
\newtheorem{Lem}{Lemma}[section]
\newtheorem{Cor}{Corollary}[section]
\newtheorem*{Cor*}{Corollary}
\newtheorem*{Def}{Definition}

\theoremstyle{definition}
\newtheorem{Rem}{Remark}[section]

\newtheorem*{Rem*}{Remark}
\numberwithin{equation}{section}

\DeclareMathOperator{\End}{End}
\DeclareMathOperator{\Hom}{Hom}

\DeclareMathOperator{\SL}{SL}

\DeclareMathOperator{\SU}{SU}
\DeclareMathOperator{\Tr}{tr}             

\DeclareMathOperator{\Id}{Id}

\DeclareMathOperator{\vol}{vol}
\DeclareMathOperator{\im}{im}

\renewcommand{\Im}{\operatorname{Im}}

\newcommand{\dvector}[1]{{\left(\begin{matrix}#1\end{matrix}\right)}}
\DeclareMathOperator{\dbar}{\bar\partial}
\DeclareMathOperator{\del}{\partial}

\newcommand{\R}{\mathbb{R}}
\newcommand{\C}{\mathbb{C}}
\newcommand{\N}{\mathbb{N}}
\newcommand{\Z}{\mathbb{Z}}
\renewcommand{\P}{\mathbb{P}}
\renewcommand{\H}{\mathbb{H}}  

\newcommand{\CP}{\mathbb{CP}}

\begin{document}

\title[A spectral curve approach to Lawson symmetric CMC surfaces of genus $2$]{A spectral curve approach to Lawson symmetric CMC surfaces of genus $2$}

\author{Sebastian Heller}




\date{\today}


\begin{abstract} 
Minimal and CMC surfaces in $S^3$ can be treated via their associated family of
flat $\SL(2,\C)$-connections.
In this the paper we parametrize the moduli space of flat $\SL(2,\C)$-connections
on the Lawson minimal surface of genus $2$ which are equivariant
with respect to certain symmetries of Lawson's geometric construction.
The parametrization uses Hitchin's abelianization procedure to write such connections explicitly in terms of flat line bundles on a complex $1$-dimensional torus.
This description is used to develop a spectral curve theory for the Lawson surface. This theory applies
as well to other CMC and minimal surfaces with the same holomorphic symmetries as the Lawson surface but different
Riemann surface structure. Additionally, we study the space of isospectral deformations of compact minimal surface
of genus $g\geq2$ and prove that it is generated by simple factor dressing.
 \end{abstract}

\maketitle

\setcounter{tocdepth}{1}
\tableofcontents


\section{Introduction}
\label{sec:intro}
The study of minimal surfaces in three dimensional space forms is among the oldest subjects in differential geometry.
 While minimal surfaces in euclidean 3-space are never compact,
 there exist compact minimal surfaces in $S^3.$
  In fact, it has been shown by Lawson \cite{L} that for every genus $g$
 there exists at least one embedded closed minimal surface in the 3-sphere.
 A slightly more general surface class is given by constant mean curvature (CMC) surfaces.
 Due to the Lawson correspondence the partial differential equations 
 describing minimal and CMC surfaces in $S^3$ can be treated in a uniform way.
Compact minimal and CMC surfaces of genus $0$ and $1$ are well-understood by now: The only CMC $2$-spheres in 
$S^3$ 
are the totally umbilic spheres
 as the Hopf differential vanishes. Furthermore, Brendle \cite{Br} has recently shown that the only embedded 
minimal torus in $S^3$ is the Clifford torus up to isometries. 
This was extended by Andrews and Li \cite{AL} who proved that the only embedded CMC tori in $S^3$ are the 
unduloidal rotational Delaunay tori.
Nevertheless, there exist compact immersed minimal and CMC tori in 
$S^3$ which are not congruent to the Clifford torus respectively to the Delaunay tori.
 First examples have been constructed by Hitchin \cite{H} via 
integrable systems methods. Moreover, all CMC tori in $S^3$ are 
constructed from algebro-geometric data defined on their associated spectral curve, see \cite{H,PS,B}.

 The study of minimal surfaces via integrable system methods is based on the associated 
$\C^*$-family of flat 
$\SL(2,{\C})$-connections $\nabla^\lambda,$ $\lambda\in \C^*.$
 Flatness of $\nabla^\lambda$ for all $\lambda$
in the spectral plane $\C^*$ is the gauge 
theoretic reformulation of the harmonic map equation.
Knowing the family of flat connections is tantamount to knowing the minimal surface, as the minimal surface is given by the gauge
between the trivial connections $\nabla^1$ and $\nabla^{-1}.$
Slightly more general, there also exists a family of flat connections associated to CMC surfaces in $S^3.$
They are given as the gauge between $\nabla^{\lambda_1}$ and $\nabla^{\lambda_2}$ for 
$\lambda_1\neq\lambda_2\in S^1\subset\C^*$
and have mean curvature $H=i\frac{\lambda_1+\lambda_2}{\lambda_1-\lambda_2}.$
In the abelian case of CMC 2-tori $\nabla^\lambda$ splits for generic $\lambda$ into a direct sum of flat connections on a line bundle and its dual. 
Therefore, the $\C^*$-family of flat $\SL(2,\C)$-connections associated to a CMC torus
is characterized by a spectral curve parametrizing the corresponding family of flat complex line bundles.
On surfaces of genus $g\geq2$ flat $\SL(2,\C)$-connections are generically irreducible and therefore they have 
non-abelian monodromy. 
In fact, every (compact) branched CMC surface of genus $g\geq2$ whose associated family of flat 
connections has abelian holonomy 
factors through a CMC torus or is a branched conformal covering of a round sphere \cite{Ge}.
Thus the abelian spectral curve theory for minimal and CMC tori are no longer applicable in the case
of compact immersed minimal and CMC surfaces of genus $g\geq2.$ 

 The aim of this paper is to develop what might be called an integrable systems theory for compact higher genus 
 minimal and CMC surfaces in $S^3$ based on its
 associated family of flat connections. The main benefit of this approach is that one can divide
 the construction and the study of minimal or CMC surfaces into three steps: 
 \begin{enumerate}[1.]
\item
 Write down (enough) flat $\SL(2,\C)$-connections
 on a given Riemann surface. 
 \item  Construct a family $\tilde\nabla^\lambda$ of flat 
 $\SL(2,\C)$-connections
 gauge equivalent (where the gauge is allowed to depend on the spectral parameter $\lambda$)  
 to a family of flat connections associated to a 
 CMC surface in $S^3.$ To ensure this, $\tilde\nabla^\lambda$ needs to be unitarizable for $\lambda\in S^1,$
 and trivial for $\lambda_1\neq\lambda_2\in S^1\subset \C^*$ and
 must have a special asymptotic behavior as $\lambda\to0.$
 \item
  Reconstruct an associated family of flat connections of a CMC surface from the gauge equivalent family. 
  \end{enumerate}
In a certain sense these steps occur in the integrable system approach to CMC tori \cite{H}. Here the 
 gauge class of a generic flat $\SL(2,\C)$-connection is determined by the holonomy of one of the eigenlines.
 The spectral curve parametrizes these holonomies and the gauge to the associated family can be determined with the 
 help of the eigenline bundle on the spectral curve.\\
 Similarly, the loop group approach to CMC surfaces put forward in
 \cite{DPW}, sometimes called the DPW method, starts with a family of holomorphic (or meromorphic) 
 $\SL(2,\C)$-connections on a Riemann surface. Typically, these connections are given by a $\lambda$-dependent 
 holomorphic (or meromorphic) $\mathfrak{sl}(2,\C)$-valued $1$-form called the DPW potential. The DPW potential has a special 
 asymptotic behavior for $\lambda\to0$ which guarantees the construction of a minimal surface as follows: a
  ($\lambda$-dependent) parallel frame for the family of holomorphic (or meromorphic) flat connections can be split
  into its unitary and positive parts by the loop group Iwasawa decomposition. 
  The unitary part is characterized by the property that it is unitary on the unit circle $S^1\subset\C^*$ and the positive 
  part extends to $\lambda=0$ in a special way.
  Then, the positive part is the gauge
  one is looking for, or equivalently, the unitary part is a ($\lambda$-dependent) parallel frame for a family of flat 
  connections associated to a minimal surface.\\
 It is well-known that every flat (smooth) $\SL(2,\C)$-connection on a compact Riemann surface is gauge equivalent
 (via a gauge which might have singularities) to a flat meromorphic connection, i.e., to a connection whose connection 
 $1$-form with respect to an arbitrary holomorphic frame is meromorphic. Nevertheless, it is impossible to
 parametrize meromorphic connections in a way such that one obtains a unique representative for every
 gauge class of flat $\SL(2,\C)$-connections.
 Therefore, the DPW potential does not need to exist for all $\lambda\in\C^*.$  Moreover, the meromorphic connections 
 (given by the DPW potential) need to be unitarizable for $\lambda\in S^1$ (i.e.,  unitary with 
 respect to an appropriate $\lambda$-dependent unitary metric). This reality condition leads to the problem of 
 computing 
 the monodromies of meromorphic connections, which cannot be done by now. The aim of this paper is to overcome 
 these problems, at least partially.

 The moduli space of flat $\SL(2,\C)$-connections on a compact Riemann surface of genus $2$ has, at its smooth 
 points, dimension $6g-6.$ There exist singular points, corresponding to reducible flat connections, which have to be 
 dealt with carefully, see \cite{G}. As we are studying holomorphic 
 families of connections (in the sense that the connection $1-$forms with respect to a fixed connection depend 
 holomorphically on $\lambda$), the moduli space needs to be equipped with a compatible complex structure. Moreover,
we need to determine the asymptotic behavior of the family of (gauge equivalence classes of) flat connections for 
$\lambda\to0.$ This seems to be difficult in the setup of character varieties, i.e., if we identify a gauge equivalence 
class of flat connections with the conjugacy class of the induced
holonomy representation of the fundamental group of the compact Riemann surface.
A more adequate picture of the 
moduli space of flat $\SL(2,\C)$-connections is given as an affine bundle over the "moduli space" of holomorphic 
structures of rank $2$ with trivial determinant. The projection of this bundle is given by taking (the isomorphism class
of) the complex anti-linear part of the connection.
This complex anti-linear part is a holomorphic structure, and for a generic flat connections it is even stable. 
Elements in a fiber of this affine bundle, which can be  represented by two 
flat connections with the same induced 
holomorphic structure, differ by a holomorphic $1$-form with values in the trace free endomorphism bundle. These 
$1$-forms are called Higgs fields and, as a
consequence of Serre duality, they are in a natural way the cotangent vectors of the moduli space of holomorphic 
structures, at least at its smooth points. 
The bundle is an affine holomorphic bundle and not isomorphic to a holomorphic vector bundle 
because it does not admit a holomorphic section. Nevertheless, by the Theorem of 
Narasimhan and Seshadri \cite{NS}, it has a smooth section (over the semi-stable part) which is given by the one 
to one correspondence between stable holomorphic structures and unitary flat connections. 

In addition to the study of the moduli spaces, we want to construct families of flat 
connections explicitly. This can be achieved by using Hitchin's abelianization, see \cite{H1,H2}.  The eigenlines of Higgs 
fields (with respect to some holomorphic structure $\dbar$)
whose determinant is given by the Hopf differential of the CMC surface are well-defined on a double covering of the 
Riemann surface. They determine points in an affine Prym variety and as line subbundles they intersect each other 
over the umbilics of the
minimal surface. Moreover, a flat connection with holomorphic structure $\dbar$ determines a meromorphic connection
on the direct sum of the two eigenlines of the Higgs field. The residue of this meromorphic connection can be 
computed explicitly, and the flat meromorphic connection is determined by algebraic geometric data 
on the double covering surface. Moreover, $\C^*$-families of flat connections can be written down in terms of a
spectral curve which double covers the spectral plane $\C^*.$ 
A double covering is needed as a holomorphic structure with
a Higgs field corresponds to two different eigenlines and these eigenlines come together at discrete spectral
 values. \\
 The spectral curve parametrizes the eigenlines of Higgs fields 
 $\Psi_\lambda\in H^0(M,K\End_0(V,\dbar^\lambda))$
 with $\det\Psi_\lambda=Q,$ where the holomorphic structure $\dbar^\lambda$ is the complex anti-linear part of the connection 
 $\nabla^\lambda.$
  In order to fix the (gauge equivalence classes of the) 
 flat connections $\nabla^\lambda$ additional spectral data are needed. They are given by anti-holomorphic 
 structures on the eigenlines, or, after fixing a special choice of a flat meromorphic connection on a line bundle in the 
 affine Prym variety, by a lift into the affine
 bundle of gauge equivalence classes of flat line bundle connections.
 Then, analogous to the case of tori, the asymptotic behavior for $\lambda\to 0$ of the family of flat connections can be 
 understood  explicitly: the spectral
curve branches over $0$ and the family of flat line bundle connections has a first order pole over $\lambda=0,$ see
Theorem \ref{asymptotic_at_0}. The spectral data must satisfy a certain reality condition imposed by the property 
that the connections $\nabla^\lambda$ are unitary for $\lambda\in S^1.$ 
 In contrast to the case of CMC tori this reality condition is hard to determine explicitly. 
Nevertheless, the reality condition is closely related to the geometry of the moduli spaces, see Theorem \ref{unitary_a}.
Once one has constructed such families of 
(gauge equivalence classes of) flat 
connections, one can construct minimal and CMC surfaces in $S^3$ by loop group factorization methods analogous to the
DPW method. It would be very interesting
to see whether these loop group factorizations can be made as explicit as in the case of tori via the eigenline 
construction of Hitchin \cite{H}.

In this paper we only carry out the details of the details of this program for the Lawson surface of genus $2.$  These methods easily generalize to the case of Lawson symmetric minimal and CMC surfaces of genus 
$2,$ i.e., those surfaces with the same holomorphic and space orientation preserving symmetries as the Lawson surface but with possibly different 
conformal type (determined by the cross ratio of the branch images of the threefold covering
 $M\to M/\Z_3\cong \P^1$). We shortly discuss this generalization in chapter \ref{A3}.
 As explained there one could in principle always exchange {\em minimal} by {\em CMC} and {\em Lawson surface} by {\em Lawson symmetric surface} within the paper.
 Moreover, the
definition of the spectral curve makes sense even in the case of a compact minimal or CMC surface of genus $g\geq2$
as long as the Hopf differential has simple zeros. In that case the asymptotic of the spectral data is analogous
to Theorem \ref{asymptotic_at_0}. The main difference to the general case is that we can describe the moduli space of
flat connections as an affine bundle over the moduli space of holomorphic structures explicitly, see Theorem 
\ref{abelianization}.

In Section \ref{holomorphic_structure_space}
we study the moduli space of those holomorphic structures of rank $2$ with trivial determinant that admit a flat 
connection
whose gauge equivalence classes are invariant under the symmetries of the Lawson surface of genus $2.$
 We show that this space is a projective line with a double point. In Section \ref{ch:hitchin_abelianization} we 
 parametrize 
 representatives of each isomorphism class in the above moduli space by using the eigenlines of special Higgs 
 fields. 
  This method is called Hitchin's abelianization. In our situation 
  the space of all eigenlines is given by the 1-dimensional square torus.
  By Hitchin's abelianization this torus double covers the moduli space
 of holomorphic structures away from the double point.
 This covering map
 will be crucial for the construction of a spectral curve later on.
 We use this description in Section \ref{sec:flat_connections}
 in order to parametrize the moduli space of flat $\SL(2,\C)$-connections whose gauge equivalence classes are invariant
  under the symmetries of the 
 Lawson surface. In Theorem \ref{abelianization} we prove an explicit 2:1 correspondence (away from a
 co-dimension $1$ subset corresponding to the double point of the moduli space of holomorphic structures)
 between flat $\C^*$-connections on the above mentioned square torus and the 
 moduli space of flat connections
whose gauge equivalence classes are invariant under the symmetries of the Lawson surface of genus $2.$
 This study will be completed in Section \ref{exceptional_connections} where we consider
 flat connections whose underlying holomorphic structures do not admit Higgs fields whose determinant is equal to the
 Hopf differential of the Lawson surface.\\
In Section \ref{sec:the_spectral_data} we define the spectral
curve associated to a minimal surface in $S^3$ which has the conformal type and the holomorphic symmetries of the 
Lawson surface of genus $2$ (Proposition \ref{The definition of the spectral curve}). The spectral curve is equipped with 
a meromorphic
 lift into the affine bundle of isomorphy classes of flat line bundle connections on the square torus. 
 This lift determines the gauge equivalence 
 classes of the flat connections.  The spectral data satisfy two important properties, see Theorem \ref{asymptotic_at_0}. 
 Firstly, they have a certain asymptotic at $\lambda=0.$ 
 Moreover, the spectral data must satisfy a reality condition which is related to 
 the geometry of the moduli space of stable bundles, see Theorem \ref{unitary_a}. 
 We prove a general theorem (\ref{general_reconstruction})
  about the reconstruction of minimal surfaces out of those families of flat connections $\tilde\nabla^\lambda$ as
  described in step $2$ above.\\
  Similar to the case of tori, compact minimal and CMC surfaces of higher genus are in general not uniquely determined
  by the knowledge of the gauge equivalence classes of $\nabla^\lambda$ for all $\lambda\in\C^*.$ 
  We show in Theorem \ref{dressing} that all different minimal immersions with
   the same map $\lambda\mapsto[\nabla^\lambda]$ into the moduli space of flat connections and with the same induced Riemann surface structure 
   are generated by simple factor dressing (as defined for example in \cite{BDLQ}).
     Finally, we prove in Theorem \ref{reconstruction_spec} that minimal
  surfaces with the symmetries of the Lawson genus $2$ surface can be  
  reconstructed uniquely from spectral data satisfying the
  conditions of Theorem \ref{asymptotic_at_0}.
  Moreover,  we give an energy formula for those minimal surfaces 
  in terms of their spectral data.
 \\
In the appendix, we shortly recall the gauge theoretic reformulation of the minimal surface equations in $S^3$ due to 
Hitchin \cite{H} which leads to the associated family of flat connections. We also describe the construction of the Lawson
minimal surface of genus $2.$ 

The author thanks Aaron Gerding, Franz Pedit and Nick Schmitt for helpful
discussions. This research was supported by the German Research
Foundation (DFG) Collaborative Research Center SFB TR 71.
 Part
of the work for the paper was done when the author was a member of the trimester
program on ÒIntegrability in Geometry and Mathematical PhysicsÓ at the Hausdorff
Research Institute in Bonn. He would like to thank the organizers for the
invitation and the institute for the great working environment.

\section{The moduli space of Lawson symmetric holomorphic structures}\label{holomorphic_structure_space}
Before studying the associated family of flat $\SL(2,\C)$-connections \[\lambda\mapsto \nabla^\lambda\] for a given compact oriented minimal or CMC surface in 
$S^3$ (see Appendix \ref{A1} or chapter \ref{A3} in the case of CMC surfaces),
we need to understand the moduli space of gauge equivalence classes of flat  $\SL(2,\C)$-connections on the surface.
We consider it as an affine bundle over the moduli space of isomorphism classes of holomorphic structures 
$(V,\dbar)$ of rank $2$ with trivial determinant over the Riemann surface. The complex structure is the one 
induced by the minimal immersion and the projection is given by taking the complex anti-linear part 
\[\nabla'':=\frac{1}{2}(\nabla+i*\nabla)\]
of the flat connection $\nabla.$ 
The difference $\Psi=\nabla^2-\nabla^1\in\Gamma(M,K\End_0(V))$ between two flat $\SL(2,\C)$-connections 
$\nabla^1$ and $\nabla^2$ with the same underlying holomorphic structure $\dbar=(\nabla^i)''$ satisfies 
\[0=F^{\nabla^2}=F^{\nabla^1}+d^\nabla\Psi=\dbar\Psi.\]
Therefore, the fiber of the affine bundle over a fixed isomorphism class of holomorphic structures (represented by the 
holomorphic structure $\dbar$)
is given by the space of Higgs fields \[H^0(M,K\End_0(V,\dbar)),\]
i.e., the space of holomorphic trace free endomorphism valued $1$-forms on $M.$ By Serre duality, this is naturally 
isomorphic to the cotangent space of the moduli space of holomorphic structures, at least at its smooth points.

 In this paper we mainly focus on the Lawson minimal surface $M$ of genus $2.$
Therefore we start by studying those holomorphic structures of rank $2$ on $M$ which can occur as the 
complex anti-linear parts of a connection $\nabla^\lambda$ in the associated family of $M.$ As we will see, this 
simplifies the study of the moduli spaces and allows us to find explicit formulas for flat connections 
with a given underlying holomorphic structure.

The complex structure of the Lawson surface of genus $2$ is given by (the compactification of) the complex curve
\begin{equation}\label{curve}
y^3=\frac{z^2-1}{z^2+1}.
\end{equation}
As a surface in $S^3$ it has a large group of extrinsic symmetries, see Appendix \ref{A2}. We will focus on the 
symmetries 
which are holomorphic on $M$ and orientation preserving in $S^3.$ The reason for this restriction relies on the fact that 
only those give rise to symmetries of the individual flat connections $\nabla^\lambda.$ 
As a group, they are generated 
by the following automorphisms, where the equations are written down with respect to the coordinates $y$ and $z$ of
\eqref{curve}:
\begin{itemize}
\item the hyper-elliptic involution $\varphi_2$ of the surface of genus $2$ which is given by
\[(y,z)\mapsto(y,-z);\]
\item the automorphism $\varphi_3$  satisfying
\[\varphi_3(y,z)=(e^{\frac{2}{3}\pi i} y,z);\]
\item the composition $\tau$ of the reflections at the spheres $S_1$ and $S_2$ is given by 
\[(y,z)\mapsto(e^{\frac{1}{3}\pi i}\frac{1}{y},\frac{i}{z}).\]
\end{itemize}

Every single connection $\nabla^\lambda$ is gauge equivalent to $\varphi_2^*\nabla^\lambda,$ 
$\varphi_3^*\nabla^\lambda$ and $\tau^*\nabla^\lambda.$  This can be deduced from the construction of the associated 
family of flat connections,
see \cite{He} for details. 
\begin{Def}
A $\SL(2,\C)$-connection $\nabla$ on $M$ is called Lawson symmetric, if $\nabla$ is gauge equivalent to 
$\varphi_2^*\nabla,$ $\varphi_3^*\nabla$ and $\tau^*\nabla.$ 
Similarly, a holomorphic structure $\dbar$ of rank $2$ with trivial determinant on $M$ is called Lawson symmetric if it is 
isomorphic to $\varphi_2^*\dbar,$ $\varphi_3^*\dbar$ and $\tau^*\dbar.$ 
\end{Def}

We first determine which holomorphic structures occur in the family
\[\lambda\mapsto \dbar^\lambda:=(\nabla^\lambda)''=\frac{1}{2}(\nabla^\lambda+i*\nabla^\lambda)\]
associated to the Lawson surface.
 As $\nabla^\lambda$ is generically irreducible (see \cite{He1}), and special unitary
for $\lambda\in S^1,$ $\dbar^\lambda$ is generically stable: A holomorphic bundle of rank $2$ of degree $0$ is 
(semi-)stable
if every holomorphic line sub-bundle has negative (non-positive) degree, see \cite{NS} or \cite{NR}.
On a compact Riemann surface of genus $ 2$ the moduli space of stable holomorphic structures with trivial determinant 
on a vector bundle of rank $r=2$ can be identified with an open dense subset of a projective $3$-dimensional space, 
see \cite{NR}: the set of those holomorphic line bundles, which are dual to a holomorphic line subbundle of degree
 $-1$ in the holomorphic rank $2$ bundle, is given by the support of a divisor 
which is linear equivalent to 
twice the $\Theta$-divisor in the Picard variety $Pic_1(M)$ of holomorphic line bundles of degree $1.$  This divisor
uniquely determines the rank $2$ bundle up to isomorphism if the bundle is stable. 
Therefore the moduli space of stable holomorphic structures of rank $2$ with trivial 
determinant can be considered as a subset of the projective space of the $4$-dimensional space 
$H^0(Jac(M),L(2\Theta))$ of $\Theta$ functions
 of rank $2$ on the Jacobian of $M.$
The complement of this subset in the projective space
is the Kummer surface associated to the Riemann surface of genus $2.$ The points on the Kummer surface can be 
identified with the S-equivalence classes of strictly semi-stable holomorphic structures.
Recall that the S-equivalence class of a stable holomorphic structure is just its isomorphism class but 
that S-equivalence identifies the strictly semi-stable holomorphic direct sum bundles $V=L\oplus L^*$ 
(where $deg(L)=0$) with 
nontrivial extensions $0\to L\to V\to L^*\to 0.$ An extension $0\to L\to V\to L^*\to 0$ 
(where $L$ is allowed to have arbitrary degree) is given by
a holomorphic structure of the form
\[\dbar=\dvector{\dbar^L & \gamma \\ 0 &\dbar^{L^*} },\]
where $\gamma\in\Gamma(M,\bar K\Hom(L^*,L)).$ It is called non-trivial if the holomorphic structure is not isomorphic to 
the holomorphic direct sum $L\oplus L^*.$ This is measured by the extension class $[\gamma]\in H^1(M,\Hom(L^*,L)).$ 
Note that the isomorphism class of the holomorphic bundle $V$ given by an extension $0\to L\to V\to L^*\to0$ with 
extension 
class $[\gamma]$ is already determined by $L$ and $\C[\gamma]\in \P H^1(M,\Hom(L^*,L)).$

\begin{Pro}\label{projective_line}
Let $\mathcal M\subset\P^3=\P H^0(Jac(M),L(2\Theta))$ be the space of S-equivalence classes of semi-stable Lawson
symmetric
holomorphic structures over the Lawson surface $M.$ 
Then the connected component $\mathcal S$ of $\mathcal M$ containing the trivial holomorphic structure
$(\underline \C^2,d'')$ is given by a projective line in $\P^3.$
\end{Pro}
\begin{proof}
The fix point set of any of these three symmetries is given by the union of projective subspaces of $\P^3.$ Clearly, the 
common fix point set contains a projective subspace of dimension $\geq1,$ as $\lambda\mapsto\dbar^\lambda$ is a 
non-constant holomorphic map into this space.

The space of S-equivalence classes of semi-stable non-stable bundles is the Kummer surface of $M$ in $\P^3.$ It has
 degree $4,$ and $16$ double points. These double points are given by extensions of self-dual line bundles 
 $L$ by itself. In order to see that $\mathcal S$ is a projective line it is enough to show that 
 the only strictly semi-stable bundles $V,$ whose isomorphism classes are 
invariant under $\varphi_2,$ $\varphi_3$ and $\tau,$ are the trivial rank two bundle $\underline\C^2$ (which is a double 
point in the Kummer surface) and the direct sum bundles 
\[L(P_1-P_2)\oplus L(P_2-P_1),\;\; L(P_1-P_4)\oplus L(P_4-P_1),\]
where $P_1,..,P_4\in M$ are the zeros of the Hopf differential of $M.$ So let $L$ be a holomorphic line sub-bundle of 
$V$ of degree $0.$ Because $M$ has genus $2$ there exists two points $P,Q\in M$ such that $L$ is given as 
the line bundle $L(P-Q)$ associated to the divisor $P-Q.$ If $P=Q$ then $V$ is in the S-equivalence class of 
$\underline\C^2.$
If $P\neq Q$ then $\varphi_2^*L(P-Q)$ is either isomorphic to $L(P-Q)$ or $L(Q-P),$ as $\varphi_2^*V$ and $V$ are 
S-equivalent. Clearly, the same holds for $\tau,$ and as $\varphi_3$ is of order $3$ we even get that
$\varphi_3^*L(P-Q)=L(P-Q).$
From these observations we  
deduce that the points $P$ and $Q$ are fixed points of $\varphi_3,$ and as a consequence $V$ is S-equivalent to
one of the above mentioned direct sum bundles.
\end{proof}
The next proposition shows that we do not need to care about S-equivalence of holomorphic bundles.
\begin{Pro}\label{semi-stable}
Every Lawson symmetric strictly semi-stable holomorphic rank $2$ bundle $V\to M$  is isomorphic to 
the direct sum of two holomorphic line bundles.
\end{Pro}
\begin{proof}
As we have seen in the proof of the previous theorem $V$ is S-equivalent to one of the holomorphic rank $2$ bundles
 $\underline \C^2,$ $L(P_1-P_2)\oplus L(P_2-P_1)$ and $ L(P_1-P_4)\oplus L(P_4-P_1).$ As
  \[\varphi_2^*L(P_i-P_j)= L(P_j-P_i)\neq L(P_i-P_j)\]
for $i\neq j$ we see that $V$ cannot be a non-trivial extension of $L(P_i-P_j)$ by its dual $L(P_j-P_i).$ It remains
to consider the case where $V$ is S-equivalent to $\underline\C^2.$ Then the holomorphic structure of $V$ is given by
\[\dvector{\dbar^\C& \gamma\\  0 & \dbar^\C}.\]
Here $\gamma\in\Gamma(M,\bar K)$ and the projective line of its cohomology class in $H^1(M,\C)$ is an invariant of 
the isomorphism class of $V.$ This projective line is determined by its annihilator in $H^0(M,K)=H^1(M,\C)^*.$ 
The annihilator of $[\gamma]$ is $H^0(M,K)$ exactly in the case where $V$ is (isomorphic to) the holomorphic direct 
sum $\C^2\to M,$
and otherwise it is a line in $H^0(M,K).$
Since $V$ is isomorphic to $\varphi_2^*V,$ $\varphi_3^*V$ and $\tau^*V$ 
this line would be invariant under $\varphi_2,$ $\varphi_3$ and $\tau$ which leads to a contradiction. 
\end{proof}

\subsection{Non semi-stable holomorphic structures}\label{non_semi_stable}
It was shown in \cite{He1} that for a generic $\lambda\in\C^*$ the holomorphic structure $\dbar^\lambda$ is stable.
Nevertheless there can exist special $\lambda\in\C^*$ such that $\dbar^\lambda$ is neither stable nor semi-stable.
We now study which non semi-stable holomorphic structures admit Lawson symmetric flat connections. 

Let $\nabla$ be a flat, Lawson symmetric $\SL(2,\C)$-connection on a complex rank $2$ bundle over $M$ such that 
$\nabla''=\dbar$ is not semi-stable. By assumption, there exists a holomorphic line subbundle  $L$ of $(V,\dbar)$ of 
degree $\geq1.$ The second fundamental form
\[\beta=\pi^{V/L}\circ\nabla_{\mid L}\in \Gamma(KHom(L, V/L))=\Gamma(KL^{-2})\] 
of $L$ with respect to $\nabla$  is holomorphic by flatness of 
$\nabla.$ As $deg(L)\geq1,$ $L$ cannot be a parallel subbundle because in that case it would inherit a flat connection. 
This implies $\beta\neq0.$ Therefore $L^{-2}=K^{-1}$ which means that $L$ is a spin bundle of $M.$ The only spin 
bundle $S$ of $M$ which is isomorphic to $\varphi_2^*S,$ $\varphi_3^*S$ and $\tau^*S$ is given by 
\[S=L(Q_1+Q_3-Q_5),\]
see \cite{He}. As there exists a flat connection with underlying holomorphic structure $\dbar,$ the bundle
$(V,\dbar)$ cannot be isomorphic to the holomorphic direct sum $S\oplus S^*\to M.$
Therefore it is given by a non-trivial extension
$0\to S\to V\to S^*\to 0.$ As $H^1(M,S^2)$ is $1$-dimensional, a non semi-stable holomorphic structure admitting a 
flat, Lawson symmetric $\SL(2,\C)$-connection is already unique up to isomorphism.
A particular choice of such a flat connection $\nabla$ is given by the uniformization connection, see \cite{H2}:
Consider the holomorphic direct sum $V=S\oplus S^*\to M,$ where $S$ is the spin bundle mentioned above. 
On $M$ there exists a unique Riemannian metric of constant curvature $-4$ in the conformal class of the Riemann 
surface $M.$ This Riemannian metric induces spin connections and unitary metrics on $S$ and 
$S^*.$ Let $\Phi=1\in H^0(M,K\Hom(S,S^*))$ and $\Phi^*=\vol$ be its dual with respect to the metric. Then
\begin{equation}\label{uniformization_connection}
\nabla^u=\nabla=\dvector{\nabla^{spin}& \vol\\  1 & \nabla^{spin^*}},
\end{equation}
is flat. Moreover, it is also Lawson symmetric. This can easily be deduced from the uniqueness of the conformal 
Riemannian metric of constant curvature $-4.$  The holomorphic structure $\nabla''$ is clearly given by the non-trivial 
extension $0\to S\to V\to S^*\to 0.$ 
\begin{Pro}\label{non_semi_stable_connection}
Every flat, Lawson symmetric connection $\nabla$ on $M,$ whose underlying holomorphic structure $\nabla''$ is not 
semi-stable, is gauge equivalent to
\[\dvector{\nabla^{spin}& C\, Q+ \vol\\  1 & \nabla^{spin^*}},\]
where $\nabla^{spin}$ and $\vol$ are the spin connection and the volume form of the conformal metric of constant 
curvature $-4$ on $M,$
$C\in\C$ and $Q$ is the Hopf differential of the Lawson surface.
\end{Pro}
\begin{proof}
Every other flat $\SL(2,\C)$-connection $\nabla,$ whose underlying holomorphic structure is $\dbar,$ is given by 
$\nabla=\nabla^u+\Psi$ where
\[ \Psi\in H^0(M,K\End_0(V,\dbar))\]
is a Higgs field. An arbitrary section $\Psi\in\Gamma(M, K\End_0(V,\dbar))$ is given by
\[\Psi=\dvector{a & b\\ c & -a}\]
with respect to the decomposition $V=S\oplus S^*$  and the matrix entries are thus sections 
$a\in\Gamma(M,K),$ $b\in\Gamma(M,K^2)$ and $c\in\Gamma(M,\C).$
Then
\[\dbar\dvector{a & b\\ c & -a}=\dvector{\dbar^K a+ c \vol & \dbar^{K^2} b+2 a \vol\\ \dbar^\C c & -\dbar^K a-c \vol}.\]
This shows that $c=0$ if $\Psi$ is holomorphic.
Moreover, for a holomorphic $1$-form $\alpha\in H^0(M,K)$ the gauge 
\[g:=\dvector{1 & \alpha\\ 0 & 1}\]
is holomorphic with respect to $\dbar$ and satisfies
\[g^{-1}\nabla^u g-\nabla^u=\dvector{-\alpha & \del^K \alpha\\ 0 & \alpha}\in H^0(M, K\End_0(V,\dbar)).\]
Therefore, we can restrict our attention to the case of Higgs fields which have the following form
\[\Psi:=\dvector{0 & b\\ 0 & 0}\]
where $b$ is a holomorphic quadratic differential by holomorphicity of $\Psi.$ The Hopf differential $Q$ of the Lawson 
surface is (up to constant multiples) the only holomorphic quadratic differential on $M$ which is 
invariant under $\varphi_2,$ $\varphi_3$ and $\tau.$ From this it easily follows that if the gauge equivalence class of
$\nabla^u+\Psi$ is invariant under $\varphi_2,$ $\varphi_3$ and $\tau$ then $b$ must be a constant multiple of $Q.$
\end{proof}
\begin{Rem}\label{two_points}
The orbits under the group of gauge transformations of the above mentioned non semi-stable holomorphic structure and
$\dbar^0$ get arbitrarily close to each other: Consider the families of holomorphic structures
\[\dbar_t=\dvector{\dbar^S & t\, vol \\ Q^* &\dbar^{S^*} } \text{ and }\tilde\dbar_t=\dvector{\dbar^S &  vol \\ t\, Q^* &\dbar^{S^*} }\]
on $V=S\oplus S^*.$ Clearly, $\dbar_t$ and $\tilde\dbar_t$ are gauge equivalent for $t\neq 0.$ On the other hand
 $\dbar_0=\dbar^0$ and  $\tilde\dbar_0=\dbar$ which are clearly not isomorphic. We will see later how to distinguish
such families of isomorphism classes of holomorphic structures if they are equipped with corresponding families of 
gauge equivalence classes of flat connections.
\end{Rem}

\section{Hitchin's abelianization}\label{ch:hitchin_abelianization}
A very useful construction for the study of a moduli space of holomorphic (Higgs) bundles is given by Hitchin's integrable 
system 
\cite{H1,H2}. We do not describe this integrable system in detail but apply some of the methods in order to construct the 
moduli space $\mathcal S,$ which was studied in the previous chapter, explicitly. The main idea is the following:
A holomorphic structure of rank $2$ equipped with a Higgs field is already determined by the eigenlines of the Higgs 
field (which are in general only well-defined on a double covering of the Riemann surface). In fact,
the rank $2$ bundle is the push forward of the dual of an eigenline bundle. In our 
situation, appropriate Higgs fields of a Lawson symmetric holomorphic structure are basically unique up to a 
multiplicative constant by Lemma \ref{exceptional_structures} and its proof. In general the two eigenlines 
are given by points in a Prym variety which are dual to each other. This Prym variety turns out to be the Jacobian of a 
$1$-dimensional square torus in the case of Lawson symmetric holomorphic structures with symmetric Higgs fields,
see Lemma \ref{reduction_to_torus} and \ref{torus_in_prym}. 
Moreover, this Jacobian double covers the moduli space $\mathcal S$ in a natural way (Proposition \ref{PI}).

\begin{Lem}\label{exceptional_structures}
Let $\dbar$ be a Lawson symmetric, semi-stable holomorphic structure on a rank $2$ bundle over $M$ which is not 
isomorphic to  $\dbar^0.$ Then there exists a Higgs field $\Psi\in H^0(M,K\End_0(V,\dbar))$ with
\[\det\Psi=Q\in H^0(M,K^2)\]
 which satisfies $\varphi^*\Psi=g^{-1}\Psi g$ for every Lawson symmetry $\varphi,$ where $g$ is the isomorphism 
 between the holomorphic structures $\dbar$ and $\varphi^*\dbar.$ This Higgs fields is unique up to sign.
\end{Lem}
\begin{proof}
By Proposition \ref{semi-stable} every Lawson symmetric, semi-stable and non-stable holomorphic structure is the 
holomorphic direct sum of two line bundles. For these bundles, it is easy to construct a Higgs field $\Psi$ with
$\det\Psi=Q.$ Moreover, this Higgs field $\Psi$ can be constructed such that its pull-back $\varphi^*\Psi$ for a Lawson 
symmetry $\varphi$ is conjugated to $\Psi.$

All stable holomorphic structures give rise to smooth points
in the moduli space of holomorphic structures. Let $[\mu]\in H^1(M,\End_0(V))$ be a non-zero tangent vector of 
the isomorphism class of the stable holomorphic structure $\dbar$ in $\mathcal S.$ By the non-abelian Hodge theory
(see for example \cite{AB})
and the Theorem of Narasimhan-Seshadri, the class $[\mu]$ can be represented by a endomorphism-valued complex 
anti-linear $1$-form $\mu\in \Gamma(M,\bar K \End_0(V))$ which is parallel with respect to the (unique) unitary flat 
connection $\nabla$ with $\nabla''=\dbar.$  Let $\varphi$ be
one of the symmetries $\varphi_2,$ $\varphi_3$ or $\tau$ and $g$ be a gauge, i.e., a smooth isomorphism, of $V$ such
that $\varphi^*\dbar=g^{-1}\dbar g.$ As $\dbar$ is stable $g$ is unique up to multiplication with a constant multiple of the
identity. We claim that $\varphi^*\mu=g^{-1}\mu g.$ To see this note that $g^{-1}\mu g$ represents (with respect to
$g^{-1}\dbar g$) the same tangent vector in $T_{[\dbar]}\mathcal S$
as $\mu$ (with respect to $\dbar$) and as $\varphi^*\mu$ (with respect to $\varphi^*\dbar=g^{-1}\dbar g$). 
Therefore
$[g^{-1}\mu g-\varphi^*\mu]=0\in T_{[\dbar]}\mathcal S,$ and by non-abelian Hodge theory
$g^{-1}\mu g-\varphi^*\mu$ is in the image of $g^{-1}\nabla g.$ 
Moreover the unitary flat connections $g^{-1}\nabla g$ and $\varphi^*\nabla$ coincide by the uniqueness of Narasimhan 
and Seshadri Theorem. Hence the difference
$g^{-1}\mu g-\varphi^*\mu$ is parallel. This is only possible if $g^{-1}\mu g-\varphi^*\mu=0$ as claimed.

Consider the (non-zero)
 adjoint $\Psi=\mu^*\in H^0(M,K\End_0(V))$ which clearly satisfies $\varphi^*\Psi=g^{-1}\Psi g$ for
$\varphi$ and $g$ as above. Therefore the holomorphic quadratic differential $\det\Psi\in H^0(M;K^2)$ is invariant under
$\varphi_2,$ $\varphi_3$ and $\tau.$ If $\det\Psi\neq0$ this implies that it is a constant non-zero multiple of the Hopf 
differential of the Lawson surface.
If  $\det\Psi=0$ consider the holomorphic  line bundle $L=\ker\Psi\subset V.$ As $\Psi$ is trace-free 
 it defines $0\neq\tilde\Psi\in H^0(M,K\Hom(V/L,L))=H^0(M,KL^2).$ Because $\deg(L)\leq-1$ as $\dbar$ is stable,
 $L$ must be dual to a spin bundle, and because $\varphi^*\Psi=g^{-1}\Psi g$
 it is even the dual of the holomorphic spin bundle $S=L(Q_1+Q_3-Q_5).$ 
 This easily implies that $\dbar$ is 
 isomorphic to $\dbar^0$
in the case of $\det\Psi=0.$
 \end{proof}
\begin{Def}
The Higgs fields of Lemma \ref{exceptional_structures} are called symmetric Higgs fields. 
\end{Def}
\subsection{The eigenlines of symmetric Higgs fields}
The zeros of the Hopf differential $Q$ of the Lawson surface $M$ are simple. As a Higgs field is trace free by definition, 
the eigenlines of a symmetric Higgs field $\Psi$ (for a Lawson symmetric holomorphic structure $\dbar$) with 
$\det\Psi=Q$ are not well-defined on the Riemann surface $M.$ Following Hitchin \cite{H1}
 we define a (branched) double covering of $M$ on which the square root of $Q$ is well-defined:
\[\pi\colon\tilde M:=\{\omega_x\in  K_x | x\in M, \omega_x^2=Q_x\}\to M.\]
We denote the involution $\omega_x\mapsto-\omega_X$ by $\sigma\colon\tilde M\to\tilde M.$
There exists a tautological section \[\omega\in H^0(\tilde M, \pi^*K_M)\] satisfying 
\[\omega^2=\pi^*Q\,\text{ and } \,\sigma^*\omega=-\omega.\]
As the Hopf differential is invariant under $\varphi_2,\ \varphi_3$ and $\tau$ these symmetries of $M$ lift to symmetries 
of $\tilde M$ denoted by the same symbols. The tautological section is invariant under these symmetries
\[\varphi_2^*\omega=\omega,\, \varphi_3^*\omega=\omega,\, \tau^*\omega=\omega,\]
where we have naturally identified $\varphi_2^*\pi^*K_M=\pi^*\varphi_2^*K_M=\pi^*K_M$ and analogous for 
$\varphi_3$ and $\tau.$ On $\tilde M$ the eigenlines of $\pi^*\Psi$ are well-defined:
$$L_{\pm}:=\ker\pi^*\Psi\mp\omega\Id.$$
Clearly, $\sigma^*L_{\pm}=L_{\mp}.$ As the zeros of $Q=\det\Psi$ are simple, $\Psi$ has a one-dimensional kernel at 
these zeros. Therefore, the eigenline bundles $L_\pm$ intersect each other of order $1$ in $\pi^*V$ at the branch points 
of $\pi.$ Otherwise said, there is a holomorphic section
\[\wedge\in H^0(\tilde M,\Hom(L_+\otimes L_-,\Lambda^2\pi^*V))\]
which has zeros of order $1$ at the branch points of $\pi.$ Thus, $\wedge$ can be considered as a constant
multiple of $\omega\in H^0(\tilde M,\pi^*K_M)$ which has also 
simple zeros exactly at the branch points of $\pi$ by construction. Because $\Lambda^2V$ is the trivial holomorphic line 
bundle, the eigenline bundles satisfy
\begin{equation}\label{eigenline_eqn}
L_+\otimes L_-=L_+\otimes\sigma(L_+)=\pi^*K_M^*,
\end{equation}
which means that $L_\pm$ lie in an affine Prym variety for $\pi.$ Recall that the Prym variety of $\pi\colon \tilde M\to M$ 
is by definition
\[Prym(\pi)=\{L\in Jac(\tilde M)\mid \sigma^* L=L^*\}.\]
After fixing the line bundle $L=\pi^*S^*,$ which clearly satisfies \eqref{eigenline_eqn}, every other 
line bundle $L^+$ satisfying \eqref{eigenline_eqn} is given by $L^+=\pi^*S^*\otimes E$ for some holomorphic line bundle
$E\in Prym(\pi).$

\subsection{Reconstruction of holomorphic rank $2$ bundles}\label{The direct image bundle}
We shortly describe how to reconstruct the bundle $V$ from an eigenline bundle $L_{+} \to \tilde M$  of a symmetric Higgs 
field $\Psi\in H^0(M,K\End_0(V))$ with non-vanishing determinant $\det\Psi\neq 0.$ This construction will be used later 
to study Lawson symmetric holomorphic connections on $\tilde M.$
First consider an open subset $U\subset M$ which does not contain a branch value of $\pi.$
The preimage $\pi^{-1}(U)\subset\tilde M$ consists of two disjoint copies $U_1\cup U_2\subset\tilde M$ of $U.$ Because
 \[\pi^*V_{\mid U_i}=(L_+\oplus L_-)_{\mid U_i}\] 
and $\sigma(L_\pm)=L_\mp,$ we obtain a basis of holomorphic sections of $V$ over $U$  which is given by the 
non-vanishing sections \[s_1\in H^0(U_1, L_+)\,\text{ and }\, s_2\in H^0(U_1, L_-)\cong H^0(U_2, L_+).\] This local basis 
of holomorphic sections in $V$ is special linear if and only if
\[\wedge(s_1\otimes s_2)=1\in H^0(U_1,\pi^*K_M\otimes L_+\otimes L_-)=H^0(U_1,\C)\]
in $U_1.$

Next we consider the case of a branch point $p$ of $\pi$:
Let $z\colon U\subset\tilde M\to\C$ be a local coordinate centered at  $p$ such that $\sigma(z)=-z$ and
$\sigma(U)=U.$ 
A local coordinate on $\pi(U)$ around $\pi(p)\in M$ is given by $y$ with $y=z^2.$
We may choose $z$ in such a way that
\[\wedge=z dy+ higher\ order\ terms\in H^0(U,\pi^*K_M),\]
where $dy\in H^0(U,\pi^*K_M)$ is the pull-back as a section and not as a $1$-form.
Let $t_1\in H^0(U,L_+), t_2=\sigma(t_1)\in H^0(U,L_-)$ be holomorphic sections without zeros such that
\[\wedge(t_1\otimes t_2)=z\in H^0(U,\C).\]
Then there are local holomorphic basis fields $s_1,s_2$ of $V\to M$ with $s_1\wedge s_2=1$ such that
$\pi^*s_1(p)=t_1(p)=t_2(p)$ and 
\begin{equation}\label{branch_frame}
t_1=\pi^*s_1-\frac{z}{2}\pi^*s_2,\, \, t_2=\pi^*s_1+\frac{z}{2}\pi^*s_2\end{equation}
in $\pi^* V,$ or equivalently
\[\pi^*s_1=\frac{1}{2}t_1+\frac{1}{2}t_2,\, \, \pi^*s_2=\frac{1}{z}t_2-\frac{1}{z}t_1.\] 
As the last equation is invariant under $\sigma$ this gives us a well-defined special linear holomorphic frame 
$\pi^*s_1,\pi^*s_2$ of $V$ over $\pi(U)\subset M.$ 

By going through the above construction carefully without a priori knowing the existence of a rank $2$ bundle  
one can construct a holomorphic 
rank $2$ bundle $V\to M$ for any line bundle $L$ in the affine Prym variety. Then one can show that this rank $2$ 
bundle has trivial determinant
and that there exists a Higgs field on $V$ whose determinant is $Q.$ See for example \cite{H1} for details on this.

\begin{Rem}
The above reconstruction is the differential geometric formulation of the sheaf theoretic push-forward
 construction $\pi_*L_\pm^*.$  \end{Rem}
\begin{Rem}\label{2:1}
Because of Lemma \ref{exceptional_structures} a generic Lawson symmetric stable bundle $V\to M$ 
 corresponds via the above construction to exactly two different line bundles $L_+$ and $L_-=\sigma(L_+).$ 
 \end{Rem}

\subsection{The torus parametrizing holomorphic structures}
The Prym variety of the double covering
 $\pi\colon\tilde M\to M$ is complex $3$-dimensional and the moduli space $\mathcal S$ of Lawson 
symmetric holomorphic structures is only $1$-dimensional. We now
determine which line bundles $L_+$ in the affine Prym variety
correspond to Lawson symmetric holomorphic structures.

Let $\dbar$ be a Lawson symmetric holomorphic structure which admits a symmetric Higgs 
field $\Psi$ whose determinant
is the Hopf differential $Q$ of the Lawson surface. 
By the definition of symmetric Higgs fields the eigenlines $L_\pm$ of $\Psi$ are isomorphic to 
$\varphi_2^*L_\pm,$ $\varphi_3^*L_\pm$ and $\tau^*L_\pm.$
Recall that the same
 is true for our base point $\pi^*S^*$ in the affine Prym variety. Therefore, it remains to determine the connected 
 component of those holomorphic line bundles $\tilde E$ of degree $0$ on $\tilde M$ whose isomorphism class 
is invariant under $\varphi_2,$ $\varphi_3$ and 
$\tau.$ The quotient \[\tilde\pi\colon\tilde M\to\tilde M/\Z_3\] of the $\Z_3$-action induced by $\varphi_3$ is a square 
torus. Moreover, $\varphi_2$ and $\tau$
induce fix point free holomorphic involutions on $\tilde M/\Z_3$ (denoted by the same symbols). They are given by 
translations. Therefore, the pull-back $\tilde E=\tilde\pi^*E$ of every line bundle $E\in Jac(\tilde M/\Z_3)$ is invariant 
under $\varphi_2,$ $\varphi_3$ and $\tau.$

In general one has to distinguish between those bundles which are pull-backs of bundles on the quotient of some 
automorphism on a 
Riemann surface and bundles whose isomorphism class is invariant under the automorphism. In our situation 
they turn out to be the same:
\begin{Lem}\label{reduction_to_torus}
Let $\tilde E$ be a holomorphic line bundle of degree $0$ on $\tilde M.$ If its isomorphism class is invariant under 
$\varphi_2,$ $\varphi_3$ and $\tau$ then $\tilde E$ is isomorphic to the pull-back $\tilde \pi^*E$ for some 
$E\in Jac( \tilde M/\Z_3).$
\end{Lem}
 \begin{proof}
 We only sketch the proof of the lemma: Consider the corresponding flat unitary connection $\nabla$ on $\tilde E.$ 
 As the isomorphism class of $\tilde E$ is invariant under $\varphi_2,$ $\varphi_3$ and $\tau$ the gauge equivalence 
 class of $\nabla$ is also invariant under $\varphi_2,$ $\varphi_3$ and $\tau.$ This gauge equivalence class is 
 determined by its (abelian) monodromy representation \[\pi_1(\tilde M)\to U(1)=S^1\subset\C.\]
 Using the symmetries $\varphi_2,$ $\varphi_3$ and $\tau$ one can easily deduce that the connection is (gauge 
 equivalent) to the pull-back of a flat connection on the torus $\tilde M/\Z_3.$
 \end{proof}
\begin{Lem}\label{torus_in_prym}
The connected component of the space of $\Z_3-$invariant line bundles in the Prym variety of $\pi\colon\tilde M\to M$ 
containing the trivial holomorphic line bundle is given by the (pull-back of the) Jacobian of the torus $\tilde M/\Z_3.$ 
\end{Lem}
\begin{proof}
Any line bundle on the torus is given by $E=L(x-p),$ where $x$ is a suitable point on the torus and $p$ is the image of the branch 
point $P_1\in\tilde M.$ The involution $\sigma$ descends to an involution on $\tilde M/\Z_3$ with four fix points
which are exactly the images of the branch points of $\tilde\pi.$ Therefore, the
quotient of $\tilde M/\Z_3$ by $\sigma$ is the projective line $\P^1$ and
\[E\otimes\sigma^* E=L(x-p+\sigma(x)-p)=\underline\C\] which implies
$\tilde\pi^*E\otimes\sigma^*(\tilde\pi^*E)=\tilde\pi^*(E\otimes\sigma^*E)=\underline\C.$
\end{proof}

These two lemmas enable us to define a double covering 
$\Pi\colon Jac(\tilde M/\Z_3)\to\mathcal S=\P^1:$
 Take a line bundle
$L\in Jac(\tilde M/\Z_3)$ and consider \[L_+:=\pi^*S^*\otimes \tilde\pi^*L\to \tilde M.\] The isomorphism class of this line 
bundle is invariant under $\varphi_2,$ $\varphi_3$ and $\tau$ and it satisfies 
\[L_+\otimes\sigma(L_+)=\pi^*K_M\]
by Lemma \ref{torus_in_prym}. As we have seen in Section \ref{The direct image bundle}, $L_+$ is an eigenline bundle 
of a symmetric Higgs field of the pullback $\pi^*V\to\tilde M$ of a 
holomorphic rank two bundle $V\to M$ with trivial determinant. 
\begin{Pro}\label{PI}
There exists an even holomorphic map 
\begin{equation}
\Pi\colon Jac(\tilde M/\Z_3)\to\mathcal S=\P^1
\end{equation}
of degree $2$
to the moduli space $\mathcal S$ of Lawson symmetric holomorphic bundles. This map is determined by
$\Pi(L)=[\dbar]$ for $L\neq\underline\C\in Jac(\tilde M/\Z_3)$ such that $\pi^*S^*\otimes \tilde\pi^*L$ is isomorphic 
to an eigenline bundle of a symmetric Higgs field of the Lawson symmetric holomorphic rank two bundle $(V,\dbar),$   and by $\Pi(\underline\C)=[\dbar^0]\in \mathcal S$ (see Lemma 
\ref{exceptional_structures}). 
The branch points are the spin bundles of 
 $\tilde M/\Z_3$ and the branch images of the non-trivial spin bundles are exactly the isomorphism classes of the
 semi-stable non-stable holomorphic bundles.
\end{Pro}
\begin{proof}
First we show that for $L\neq\underline\C,$ the corresponding
rank two bundle is semi-stable: Assume that $E$ is a holomorphic line subbundle of 
a Lawson symmetric holomorphic bundle $V\to M$ of degree greater than $0.$ If
$E$ is not a spin bundle of the genus $2$ surface $M$ the rank two bundle $V$ would be isomorphic to the holomorphic 
direct sum $E\oplus E^*.$ In this case one easily sees that there do not exists a Higgs field whose determinant has 
simple zeros. If $E$ is a spin bundle it must be isomorphic to the spin bundle $S$ of the Lawson immersion because of
the symmetries. Let the rank two holomorphic structure be given by
\[\dbar=\dvector{\dbar^S & \alpha \\ 0 & \dbar^{S^*}}\] on the topological direct sum $V=S\oplus S^*$ for some
$\alpha\in\Gamma(M,\bar KK).$
The eigenline bundle $\pi^*S^*\otimes \tilde\pi^*L\subset V$ would be given by a map
 \[\dvector{a\\b}\colon\pi^*S^*\otimes \tilde\pi^*L\to\pi^*S\oplus\pi^*S^*\]
satisfying $\dbar a+\alpha b=0$ and $\dbar b=0.$
For $L\neq\underline\C\in Jac(\tilde M/\Z_3)$ there does not exists a holomorphic map from
$\pi^*S^*\otimes \tilde\pi^*L$ to $\pi^*S^*.$ Therefore the eigenline bundle $\pi^*S^*\otimes \tilde\pi^*L$ would be
$\pi^*S,$ which is impossible because of the degree. Moreover one easily sees that the 
corresponding holomorphic rank two bundle for $L=\underline\C$ must be isomorphic to the holomorphic direct sum 
$V=\pi^*S\oplus\pi^*S^*.$
The orbit of this holomorphic structure under the gauge group is infinitesimal near to the one of the holomorphic 
structure $\dbar^0$ of Lemma 
\ref{exceptional_structures}. Therefore
we can map $\underline\C\in Jac(\tilde M/\Z_3)$ to the equivalence class of the stable holomorphic structure 
 $\dbar^0$ in $\mathcal S=\P^1$ in order to obtain a well-defined holomorphic map 
 $\Pi\colon Jac(\tilde M/\Z_3)\to\mathcal S.$
 
Because of Lemma \ref{exceptional_structures} and remark \ref{2:1} the degree of the map $\Pi$ is $2.$
 Clearly $\Pi(L)=\Pi(L^*)$ for all $L\in Jac(\tilde M/\Z_3).$ Therefore the spin bundles of $\tilde M/\Z_3$ are the only
  branch points of $\Pi.$
It remains to show that the non-trivial spin bundles in $Jac(\tilde M/\Z_3)$ correspond to the strictly semi-stable 
bundles $V\to M.$ This can either be seen by analogous methods as in \cite{H1} used for the computation of the
unstable locus in the Prym variety, or more directly as follows: Consider for example the non stable semi-stable bundle
$V=\underline \C\oplus\underline \C.$ Then, a symmetric Higgs field is given by
\[\Psi=\dvector{0 & \omega_1\\ \omega_2&0},\]
where $\omega_1$ and $\omega_2$ are holomorphic differentials with simple zeros at $P_1$ and $P_3$ respectively
$P_2$ and $P_4$ such that $Q=\omega_1\omega_2.$
Then the eigenlines $\ker(\Psi\pm\alpha)$ are both isomorphic to $L(-P_1-P_3)=\pi^*S^*\otimes L(3P_1-3P_3).$
Clearly, $L(3P_1-3P_3)=\tilde\pi^*L(\tilde\pi(P_3)-\tilde\pi(P_1)),$ and $L(\tilde\pi(P_3)-\tilde\pi(P_1))$ is a non-trivial
spin bundle of $\tilde M/\Z_3.$
Therefore, the gauge orbit of the trivial holomorphic rank $2$ bundle $\underline\C^2\to M$ is a branch image of $\Pi,$
and similarly one can show that the same is true for the remaining two semi-stable non-stable holomorphic structures.
\end{proof}

\begin{Rem}\label{comparison_with_tori}
This double covering of the moduli space $\mathcal S$ of Lawson symmetric holomorphic rank two bundles is very 
similar to the one of the moduli space of holomorphic rank two bundles with trivial determinant on a Riemann 
surface $\Sigma$ of genus $1.$ The later space consist of all bundles of the form $L\oplus L^*$ where
 $L\in Jac(\Sigma)$ 
together with the non-trivial extensions of the spin bundles of $\Sigma$ with itself, see \cite{A}.
\end{Rem}

\section{Flat Lawson symmetric $\SL(2,\C)$-connections}\label{sec:flat_connections}
We use the results of the previous chapter to study the moduli space of flat Lawson symmetric connections on 
$M$ as an affine bundle over the moduli space of Lawson symmetric holomorphic structures.
A similar approach was used by Donagi and Pantev \cite{DP} in their study of the geometric Langlands correspondence.
 
 The underlying holomorphic structure 
 $\nabla''$ of a flat Lawson symmetric connection $\nabla$ is determined by a holomorphic line bundle
   $L\in Jac(\tilde M/\Z_3)$ (Proposition \ref{PI}).
   Conversely, for all non-trivial holomorphic line bundles $L\in Jac(\tilde M/\Z_3)$
   there exists a Lawson symmetric holomorphic structure which is semi-stable.
Because of the Theorem of Narasimhan and Seshadri \cite{NS}, these holomorphic structures admit flat unitary 
connections,
 and, because of the uniqueness part in \cite{NS}, the gauge equivalence class of the
 flat unitary connection is also invariant under $\varphi_2,$ $\varphi_3$ and $\tau.$ 
 In order to obtain all flat Lawson symmetric 
connections we only need to add symmetric Higgs fields to the unitary connections. 
We will see in Theorem \ref{abelianization} that flat Lawson symmetric connections on $M$ are uniquely and explicitly 
determined by a flat connection on the corresponding line bundle $L\in Jac(\tilde M/\Z_3)$ as long as
$L$ is not isomorphic the trivial holomorphic bundle $\underline\C.$
Adding a symmetric Higgs field on the Lawson symmetric connection on $M$ 
is equivalent to adding a holomorphic $1$-form to the line bundle connection on $L\to\tilde M/\Z_3.$
Therefore the affine bundle  structure of the space of gauge equivalence classes of flat Lawson symmetric connections 
on $M$ is determined by the affine bundle structure of the moduli space of flat line bundle connections over the 
Jacobian of the torus $\tilde M/\Z_3.$
 The case of the remaining flat Lawson symmetric connections (corresponding
 to the holomorphic structures which are either isomorphic to $\dbar^0$ or to the non-trivial extension
 $0\to S\to V\to S^*\to 0$) is dealt with in the next chapter. We will see that they 
   occur as special limits as $L$ converges to the trivial holomorphic line bundle. 

Let $\nabla$ be a flat Lawson symmetric connection such that its underlying 
holomorphic structure $\nabla''$ admits a symmetric Higgs field $\Psi\in H^0(K,\End_0(V))$ with $\det \Psi=Q.$
Equivalently, there is a non-trivial holomorphic line bundle $L\in Jac(\tilde M/\Z_3)$ with $\Pi(L)=[\nabla''].$
 Consider the pull-back connection $\pi^*\nabla$ 
on $\pi^*V\to\tilde M,$ where $\pi\colon \tilde M\to M$ is as in the previous chapter.
As the eigenline bundles $L_\pm\to\tilde M$ of $\pi^*\Psi$ are holomorphic subbundles of $\pi^*V,$ which only intersect
at the branch points of $\pi,$ there exists a holomorphic homomorphism 
\[f\colon L_+\oplus L_-\to\pi^*V\] which is an isomorphism away from the branch points of $\pi.$ Therefore
there exists a unique meromorphic flat connection $\tilde\nabla$ on $L_+\oplus L_-\to\tilde M$
such that $f$ is parallel.
 The poles of $\tilde\nabla$ are at the branch points of $\pi.$ 
 Let $z$ be a holomorphic coordinate on $\tilde M$ centered at a branch point $P_i$ of $\pi$ such that $\sigma(z)=-z.$
 Let $s_1,s_2$ be a special linear frame of $V$ and let $t_1$ and $t_2=\sigma(t_1)$ be local holomorphic sections in 
 $L_+$ and $L_-$ satisfying \eqref{branch_frame}.
 The connection $\nabla$ on $V\to M$ is determined locally by
\[\nabla s_j=\omega_{1,j} s_1+\omega_{2,j} s_2\]
for $j=1,2,$ where $\omega_{i,j}$ are the locally defined holomorphic $1-$forms. As $\nabla$ and the frame 
are special linear $\omega_{1,1}=-\omega_{2,2}$ holds.  Because $\pi$ has a branch point at $P_i,$
 the connection $1$-forms $\pi^*\omega_{i,j}$ (of $\pi^\nabla$ with respect to $\pi^*s_1,\pi^*s_2$) have zeros
 at $P_i.$ Using \eqref{branch_frame} one can compute the connection $1-$forms of 
 $\tilde\nabla$ with respect to the frame $t_1, t_2=\sigma(t_1)$ of $L_+\oplus L_-.$ It turns out that
 they have first order poles at $P_i.$ Moreover, 
 the residue of $\tilde\nabla$ at $P_i$ is given by
 \begin{equation}\label{residuum}
 res_{P_i}\tilde\nabla=\frac{1}{2} \dvector{1 & -1\\ -1 & 1}
 \end{equation}
 with respect to the frame $t_1, t_2.$ 
 We need to interpret this formula more invariantly. With respect to the direct sum decomposition
 $L_+\oplus L_-$ the connection $\tilde\nabla$ splits
 \begin{equation}\label{split_connection}
 \tilde\nabla=\dvector{\nabla^{+} & \beta^- \\ \beta^+ & \nabla^{-} }.
\end{equation} 
 Here, $\nabla^\pm$ are meromorphic connections on $L_\pm$ with simple poles at the branch points of $\pi,$ and
 $\beta^\pm\in\mathcal M(\tilde M, K_{\tilde M}\Hom(L_\pm, L_\mp))$ are the meromorphic second fundamental forms of 
 $L_\pm$ which also have simple poles at the branch points. Recall that the eigenline bundles are given by
 \begin{equation}\label{line_bundle}
 L_\pm=\pi^*S^*\otimes\tilde \pi^* L^{\pm 1}
 \end{equation}
 for holomorphic line bundles $L^{\pm 1}\in Jac(\tilde M/\Z_3)$ and $\tilde\pi\colon \tilde M\to \tilde M/\Z^3.$    
 Consider the holomorphic section $\wedge\in H^0(\tilde M,\pi^*K_M)$ which has simple zeros at the branch points of 
 $\pi.$ There exists an unique meromorphic connection $\nabla^{K_M}$ on $\pi^*K_M$ such that $\wedge$ is parallel. 
 Then $res_{P_i}\nabla^{K_M}=-1$ at the branch points $P_1,..,P_4.$ As $\pi^*S^2=\pi^*K_M$ there exists a unique 
 meromorphic connection $\nabla^{S^*}$ on $\pi^*S^*$ which has simple poles at the branch points of $\pi$ with 
 residue $\frac{1}{2}.$ Using \eqref{line_bundle}, the description of $\nabla^{S^*}$ and \eqref{residuum} 
 we obtain holomorphic connections $\tilde\nabla^\pm$ on $\tilde\pi^*L^+$ and $\tilde\pi^*L^{-1}$ satisfying the formula
\[\nabla^\pm= \nabla^{S^*}\otimes \tilde\nabla^\pm.\]
Moreover, $\tilde\nabla^\pm$ are dual to each other. As in the proof of Lemma \ref{reduction_to_torus} one can show
that $\tilde\nabla^\pm$ are invariant under $\varphi_2,$ $\varphi_3$ and $\tau$ and that there exists holomorphic 
connections $\nabla^{L^\pm}$ on $L^\pm\to \tilde M/\Z^3$ such that 
\[\tilde\nabla^\pm=\tilde\pi^*\nabla^{L^\pm}.\]
Then, all holomorphic connections on $L=L^+\to \tilde M/\Z^3$ with its fixed holomorphic structure are given by 
$\nabla^{L^+}+\alpha$ for a holomorphic  $1-$form $\alpha\in H^0(\tilde M/\Z^3,K_{\tilde M/\Z^3}).$ 
Clearly, the corresponding effect on the connection 
$\nabla$ on  $V\to M$ is given by the addition of (a multiple of) the symmetric Higgs field $\Psi$ which diagonalizes on 
$\tilde M$ with eigenlines  $L_\pm.$ 

\subsection{The second fundamental forms}\label{The second fundamental forms}
Next, we compute the second fundamental forms $\beta^\pm\in\mathcal M(\tilde M, K_{\tilde M}\Hom(L_\pm, L_\mp))$ 
of the eigenlines of the symmetric Higgs field. We fix some notations first: The symmetries
$\varphi_2$ and $\tau$ of $\tilde M$ yield fix point free symmetries on the torus $\tilde M/\Z_3$ denoted by
the same symbols. The quotient by these actions is again a square torus, 
 denoted by $T^2$, which is fourfold covered by $\tilde M/\Z_3$ and the corresponding map is denoted by
 \[\pi^T\colon\tilde M/\Z_3\to T^2.\]
 Each $L\in Jac(\tilde M/\Z_3)$ is the pull-back of a line bundle 
 $\hat L$ of degree $0$ on $T^2.$ This line
  bundle is not unique. Actually, the pullback map defines a fourfold covering
 \[Jac(T^2)\to Jac(\tilde M/\Z_3).\]
 In particular, there are four different line bundles on $T^2$ which pull-back to the trivial one on $\tilde M/\Z_3.$ These
  are exactly the spin bundles  on $T^2,$ so their square is the trivial holomorphic bundle.
 This implies, that $\hat L^{\pm2}$ is independent of a choice $\hat L\in Jac(T^2)$ which pulls back to a given
$L\in Jac(\tilde M/\Z_3).$ Let $0\in T^2$ be the (common) image of the branch points $P_i$ of $\pi.$
Then every holomorphic line bundle $E\to T^2$ of degree $0$ is (isomorphic to) the line bundle $L(y-0)$ associated to
an divisor of the form $D=y-0$ for some $y\in T^2.$ In particular, for $y\neq0$ there exists a meromorphic section
$s_{y-0}\in\mathcal M(T^2,E)$ with divisor $(s_{y-0})=D.$
Moreover, $y$ is uniquely determined by $E$ and $s_{y-0}$ is unique up to a multiplicative constant.
\begin{Pro}\label{comp_beta}
Let $\nabla$ be a flat Lawson symmetric connection on $M.$ 
Let $L_+=L\in Jac(\tilde M/\Z_3)$ be a non-trivial holomorphic line bundle
which is given by $L=(\pi^T)^*L(x-0)$ for some $x\in T^2$ 
 such that $\Pi(L)=[\nabla''].$  
Then the point $y:=-2x\in T^2$ is not $0$ and the second fundamental form of $L_+$ is
\[\beta^+=\tilde\pi^*(\pi^T)^* s_{y-0}\in\mathcal M(\tilde M, K_{\tilde M}\Hom(L_\pm, L_\mp))=\mathcal M(\tilde M, K_{\tilde M}\tilde\pi^*L^{\mp2})\]
where for
\[s_{y-0}\in\mathcal M(T^2,L(y-0))=\mathcal M(T^2,K_{T^2}L(y-0))\]
the multiplicative constant is  chosen appropriately 
and the pullbacks are considered as pullbacks of (bundle-valued) $1$-forms. If we denote $y^-=-y=2x\in T^2,$ then the
 second fundamental form $\beta^-$ is given by $\beta^-=\tilde\pi^*(\pi^T)^* s_{y^--0}.$
\end{Pro}
\begin{proof}
By assumption $L=(\pi^T)^*L(x-0)$ is not the trivial holomorphic line bundle. Therefore, $L(x-0)$ cannot be a spin bundle 
of $T^2.$ Equivalently, $L(x-0)^{-2}=L(y-0)$ is not the trivial holomorphic line bundle which implies that $y\neq0.$

The gauge equivalence class of the connection $\nabla$ is invariant under the symmetries. Therefore, the set of poles 
and the set of zeros of the second fundamental forms $\beta^\pm$ of the eigenlines of the symmetric Higgs field are 
fixed under the symmetries, too. There are exactly $4$ simple poles of $\beta^+$ and because $\tilde M$ has genus $5$
and the degree of $\Hom(L_\pm, L_\mp))=\tilde\pi^*L^{\mp2}$ is $0$ there are $12$ zeros of $\beta^+$ counted with 
multiplicity. The only fix points of $\varphi_3$ are the branch points $P_i$ of $\pi$ and $\varphi_2$ and $\tau$ are fix 
point free on $\tilde M.$ Therefore, the orbit of a zero of $\beta^+$ under the actions of $\varphi_2,$ $\varphi_3$ and 
$\tau$ consists of exactly $12$ points. This implies that the zeros of $\beta^+$ are simple. Moreover, these $12$ 
points are mapped via $\pi^T\circ\tilde\pi$ to a single point $\tilde y$ in $T^2.$  We claim that $\tilde y=y\in T^2.$ To see 
this, we consider the (bundle-valued) meromorphic $1$-form $\tilde\pi^*(\pi^T)^* s_{\tilde y-0}$ on $\tilde M,$ which has 
simple poles exactly at the branch points $P_i$ of $\pi$ and simple zeros at the preimages of $\tilde y.$ Therefore,
$\tilde\pi^*(\pi^T)^* s_{\tilde y-0}$ is (up to a multiplicative constant) the second fundamental form $\beta^+.$ As
 the bundle $L(y-0)$ is uniquely determined by $L_+$ we also get $\tilde y=y.$
\end{proof}
 \begin{Rem}\label{no_flat_connection}
 In the case of $y=0\in T^2$ there is no meromorphic section in the trivial line bundle $L(y-0)=\underline\C$ 
 with a simple pole at $0.$
 But $y=0$ holds exactly for the trivial 
 bundle $\underline\C\in Jac(\tilde M/\Z_3).$ This line bundle corresponds to the non-stable holomorphic direct sum 
 bundle $S^*\oplus S\to M,$ see the proof of Proposition \ref{PI}.
 As we have seen in Section
 \ref{non_semi_stable} there does not exists a holomorphic connection on $S^*\oplus S\to M.$
 \end{Rem}

By now, we have determined the second fundamental forms up to a constant.
It remains to determine the exact multiplicative constant of \[\hat\gamma^\pm:=s_{y^\pm-0}.\] Note that the involution $\sigma$ on 
$\tilde M$ gives rise
 to involutions on $\tilde M/\Z_3$ and $T^2,$ denoted by the same symbol. Then, $\sigma(\beta^\pm)=\beta^\mp$ and
  $\sigma(\hat\gamma^\pm)=\hat\gamma^\mp.$
From Equations \ref{residuum} and \ref{split_connection} one sees that 
\[\beta^+\beta^-\in\mathcal M(\tilde M, K_{\tilde M}^2)\]
is a well-defined meromorphic quadratic differential with double poles at the branch points $P_1,..,P_4$ and with 
residue
 \[res_{P_i}(\beta^+\beta^-)=\frac{1}{4}.\]
 As the branch order of $\tilde\pi$ at $P_i$ is $2$ we have
 \begin{equation}\label{res_1_36}
  res_{0}(\hat\gamma^+\hat\gamma^-)=\frac{1}{36}.
  \end{equation}
 Together with $\sigma(\hat\gamma^\pm)=\hat\gamma^\mp$ this completely determines $\hat\gamma^\pm$ and 
 therefore 
 also $\beta^\pm$ up to sign. Note that the sign has no invariant meaning as the sign of the off-diagonal terms of the 
 connection can be changed by applying a diagonal gauge with entries $i$ and $-i.$

 \subsection{Explicit formulas}
 We are now going to write down explicit formulas for a flat Lawson symmetric connection $\nabla$ whose underlying 
 holomorphic structure admits a symmetric Higgs field $\Psi$ with $\det\Psi=Q.$
  To be precise, we compute the connection $1$-form of $\pi^*\nabla\otimes\nabla^S$ with respect to some frame,
  where $\nabla^S$ is defined as above by the equation $(\nabla^S\otimes\nabla^S) \omega=0$ for the tautological 
  section $\omega\in H^0(\tilde M,\pi^*K_M).$
 Then $\pi^*\nabla\otimes\nabla^S$ is
 a meromorphic connection on $\tilde\pi^*L^+\oplus\tilde\pi^*L^-\to\tilde M$ with simple, off-diagonal poles 
  at the branch points $P_1,..,P_4$ of $\pi,$ where $L^+$ and $L^-$ are holomorphic line bundles of degree $0$ on the 
  torus $\tilde M/\Z_3$ which are dual to each other and correspond to the eigenlines of the symmetric Higgs field via
  Proposition \ref{PI}.
 
  Recall that $\tilde M/\Z_3$ is a square torus, and we identify it as
 \[\tilde M/\Z_3\cong \C/(2\Z+2i\Z).\]
 We may assume without loss of generality that the half lattice points are exactly the images of the branch points $P_i.$
 The fourfold (unbranched) covering map $\pi^T$ gets into the natural quotient map
 \[\pi^T\colon\tilde M/\Z_3\cong \C/(2\Z+2i\Z)\to \C/(\Z+i\Z)\cong T^2.\]
 Let $E$ be one choice of a holomorphic line bundle on $T^2$ which pulls 
 back to $L^+\to \tilde M/\Z_3.$ As before, it is given by
 $E=L([x]-[0])$
 for some $[x]\in T^2,$ where $[0]\in\C/(\Z+i\Z)\cong T^2$ is the common image of the points $P_i.$

The following lemma is of course well-known. We include it as it produces the trivializing sections which we use to write 
down the connection $1$-form. 
 \begin{Lem}\label{trivializing_section}
 Consider the square torus $T^2=\C/(\Z+i\Z)$ and the holomorphic line bundle $E=L([x]-[0])$ for some $x\in\C.$
  Then there exists a smooth section 
 $\underline 1\in\Gamma(T^2,E)$ such that the holomorphic structure $\dbar^E$ of $E$ is given by
 \[\dbar^E\underline 1=-\pi xd\bar z\underline1.\]
 \end{Lem}
 \begin{proof}
 The proof is merely included to fix our notations about the $\Theta$-function of $T^2=\C/(\Z+i\Z),$ see 
 \cite{GH} for details.
 There exists an even entire function $\theta\colon\C\to\C$ which has simple zeros exactly at the lattice points $\Z+i\Z$ 
 and which satisfies
 \begin{equation}\label{def_theta}
 \begin{split}
 \theta(z+1)&=\theta(z)\\
 \theta(z+i)&=\theta(z)\exp(-2\pi i (z-\frac{1+i}{2})+\pi).
 \end{split}
 \end{equation}
 Then the function
 \[s(z):=\frac{\theta(z-x)}{\theta(z)}\exp(\pi x(\bar z-z))\]
is doubly periodic and has simple poles at the lattice points $\Z+i\Z$ and simple zeros at $x+\Z+i\Z.$ Moreover
it satisfies
$\dbar s=\pi x s.$ Therefore, $s$ can be considered as a meromorphic section with respect to the holomorphic structure
$\dbar-\pi xd\bar z$ on $T^2=\C/(\Z+i\Z)$ with simple poles at $[0]\in T^2$ and simple zeros at $[x]\in T^2.$ This implies
that the holomorphic structure $\dbar-\pi x d\bar z$ is isomorphic to the holomorphic structure of $E=L([x]-[0]).$
The image $\underline 1$ of the constant function $1$ under this isomorphism satisfies the required equation
$\dbar^E\underline 1=-\pi xd\bar z\underline1.$
  \end{proof}

The second fundamental forms $\beta^\pm=\tilde\pi^*(\pi^T)^*\hat\gamma^\pm$ can be written down in terms
 of $\Theta$-functions as follows: From Proposition \ref{comp_beta} and the proof of Lemma 
 \ref{trivializing_section} one obtains that (with respect to the smooth trivializing section $\underline 1$ of $E=L([x]-[0])$ 
 and its dual section $\underline 1^* \in\Gamma(T^2,E^*)$) $\hat\gamma^\pm$ are given by
 \begin{equation}\label{explicit_gamma}
 \begin{split}
 \hat\gamma^+(z)\underline 1&=c\frac{\theta(z-y)}{\theta(z)}e^{-2\pi i y\Im(z)}\underline 1^* dz\\
  \hat\gamma^-(z)\underline 1^*&=c\frac{\theta(z+y)}{\theta(z)}e^{2\pi i y\Im(z)}\underline 1dz\\
 \end{split}
 \end{equation}
 for some $c\in\C,$ where $\theta$ is as in the proof of \ref{trivializing_section} and $y=-2x.$ The constant $c\in\C$
 is given by \eqref{res_1_36} as a choice of a square root
 \begin{equation}\label{c_sign}
 c=\frac{1}{6}\sqrt{\frac{\theta'(0)^2}{\theta(y) \theta(-y)}},
 \end{equation}
 where $'$ denotes the derivative with respect to $z.$
  \begin{Rem}
 Note that $c$ can be considered as a single-valued meromorphic function depending on $y\in\C$ with simple poles at
 the integer lattice points by choosing the sign of the square root at some given point $y\notin\Z+i\Z.$
 \end{Rem}
 Altogether, the connection $\pi^*\nabla\otimes\nabla^S$ is given on $T^2=\C/\Z+i\Z$ with respect to the frame 
 $\underline 1,\, \underline 1^*$ by the connection $1$-form
 \begin{equation}\label{connection1form}
  \dvector{\pi a dz-\pi xd\bar z & c\frac{\theta(z+y)}{\theta(z)}e^{2\pi i y\Im(z)} dz \\ c\frac{\theta(z-y)}{\theta(z)}e^{-2\pi i y\Im(z)} dz & -\pi a dz+\pi xd\bar z  }
 \end{equation}
 for some $a\in\C.$ The connection $1$-form \ref{connection1form} is only meromorphic, but the corresponding 
 connection $\nabla$ on the rank 2 bundle over $M$ has no singularities.
Varying $a\in\C$ corresponds to adding a multiple of the symmetric Higgs field on the connection 
 $\nabla.$  
 \begin{Rem}
 In \eqref{connection1form} we have written down the connection $1$-forms on the torus $T^2\cong \C/(\Z+i\Z).$
 But as the fourfold covering $\tilde M/\Z_3\cong \C/(2\Z+2i\Z)\to T^2\cong \C/(\Z+i\Z)$ is simply given by
 \[z\mod 2\Z+2i\Z \, \longmapsto \, z\mod \Z+i\Z\]
 \eqref{connection1form} gives also the connection $1$-form for the connection $\pi^*\nabla\otimes\nabla^S$ on 
 $\tilde M/\Z_3$ with respect to the frame $(\pi^T)^*\underline 1,\, (\pi^T)^*\underline 1^*.$
 \end{Rem}

We summarize our discussion:
 \begin{The}[The abelianization of flat $\SL(2,\C)$-connections]\label{abelianization}
 Let $\dbar$ be a Lawson symmetric semi-stable holomorphic structure on a rank $2$ vector bundle over $M.$
  Assume that $\dbar$ is determined by the non-trivial holomorphic line bundle $L\in Jac(\tilde M/\Z_3),$
 i.e., $\Pi(L)=[\dbar].$
  Then there is a 1:1 correspondence between 
   holomorphic connections on $L\to \tilde M/\Z_3$ and
  flat Lawson symmetric connections $\nabla$ with $\nabla''=\dbar.$ The correspondence is given explicitly by the connection $1$-form \eqref{connection1form}.
  \end{The}

 \subsection{Flat unitary connections} \label{sec:unitary_connections}
A famous result due to  Narasimhan and Seshadri (\cite{NS}) states that for every stable holomorphic structure on a 
complex vector bundle over a compact Riemann surface there exists a unique flat connection which is unitary with 
respect to a suitable chosen metric and whose underlying holomorphic structure is the given one. From the uniqueness 
we observe the following: If the isomorphism class of a a stable holomorphic structure $\dbar$ is invariant under some
automorphisms of the Riemann surface then the gauge equivalence class of the unitary flat connection $\nabla$ with 
$\dbar=\nabla''$ is also invariant under the same automorphisms.
 We apply this to the 
situation of Theorem \ref{abelianization}. 
\begin{The}\label{unitary_a}
Consider a Lawson symmetric holomorphic structure $\dbar$ of rank $2$ on $M$ whose isomorphism class
is given by a non-trivial holomorphic line bundle $L\in Jac(\tilde M/\Z_3),$ i.e., $\Pi(L)=[\dbar].$ 
Let $x\in \C\setminus (\frac{1}{2}\Z+\frac{1}{2}i\Z)$ such that
the holomorphic structure of $E$ is given by \[\dbar^E=\dbar^0-\pi x d\bar z\]
 on $\C\to \tilde M/\Z_3\cong \C/(2\Z+2i\Z).$ 
 Then there exists a unique $a^u=a^u(x)\in\C$ such that the flat Lawson symmetric connection $\nabla$ on $M$ 
 which is given by the connection $1$-form \ref{connection1form} is unitary with respect to a suitable chosen metric.
The function
\[x\mapsto a^u(x)\] is real analytic and odd in $x.$  It satisfies
\[a^u(x+\frac{1}{2})=a^u(x)+\frac{1}{2}\] and
\[a^u(x+\frac{i}{2})=a^u(x)-\frac{i}{2}\]
which means that it gives rise to a well-defined real analytic section $\mathcal U$ of the affine bundle of 
(the moduli space of) flat $\C^*$-connections over the Jacobian of $\tilde M/\Z_3$ away from the origin.
\end{The}
  \begin{Rem}
 We show in Theorem \ref{extension_dbar0} below that the
 section $\mathcal U$ has a first order pole at the origin.
 \end{Rem}
\begin{proof}
As the unitary flat connections depend (real) analytic on the underlying holomorphic structure, the function
$x\mapsto a^u(x)$ is also real analytic. Moreover, it must be odd in $x$ as the flat connection induced on 
$L^+\to\tilde M/\Z_3$ is dual to the one induced on $L^-\to\tilde M/\Z_3.$ The functional equations are simply
a consequence of the gauge invariance of our discussion: On $\tilde M/\Z_3=\C/(2\Z+2i\Z)$ the flat connections
$d+\pi a dz-\pi x d\bar z,$  $d+\pi (a-\frac{1}{2}) dz-\pi (x-\frac{1}{2})d\bar z$ and 
$d+\pi (a+\frac{i}{2}) dz-\pi (x-\frac{i}{2})d\bar z$ are gauge equivalent as well as the
corresponding flat $\SL(2,\C)$-connections on $M.$
\end{proof}
\begin{Rem}
 The Narasimhan-Seshadri section which maps an isomorphism class of stable holomorphic structures to
 its corresponding gauge class of unitary flat connections is a real analytic section in the holomorphic 
 affine bundle of the moduli space of flat $\SL(2,\C)$ connections to the moduli space of stable holomorphic
 structures. The later space is equipped with a natural symplectic structure. Then, the natural (complex anti-linear)
 derivative of the Narasimhan-Seshadri section can be interpreted as the symplectic form, see for example \cite{BR}.
 \end{Rem}

\section{The exceptional flat $\SL(2,\C)$-connections}\label{exceptional_connections}
In the previous chapter we have studied all flat Lawson symmetric connections on $M$ whose underlying holomorphic 
structures admit symmetric Higgs fields $\Psi$ such that $\det \Psi=Q.$ The holomorphic structures are determined by
non-trivial holomorphic line bundles $L\in Jac(\tilde M/\Z_3),$ see Proposition \ref{PI}. The construction of a connection 
$1$-form as in \eqref{connection1form} breaks down for the trivial holomorphic line bundle 
$\underline\C\to \tilde M/\Z_3,$ because the trivial line bundle corresponds to the holomorphic direct sum bundle
$S\oplus S^*\to M$ which does not admit a holomorphic connection. But as we have already mentioned above,
the gauge orbits of the remaining 
holomorphic structures which admit Lawson symmetric holomorphic connections are infinitesimal near to the gauge 
orbit of $S\oplus S^*\to M$ (see for example the proof of Proposition \ref{PI}). We use this observation to construct the 
remaining flat Lawson symmetric connections as limits of the connections
studied in Theorem \ref{abelianization} when $L$ tends to the trivial holomorphic line bundle. Even more important for our 
purpose, we exactly determine for which meromorphic family of flat line bundle connections on $\tilde M/\Z_3$  the 
corresponding family of flat $\SL(2,\C)$-connections on $M$ extends holomorphically through the points where the 
holomorphic line bundle is the trivial one, see Theorem \ref{extension_dbar0} and Theorem \ref{extension_non_stable} 
below.
\subsection{The case of the stable holomorphic structure}\label{The case of the stable holomorphic structure}
We start our discussion with the case of a Lawson symmetric stable holomorphic structure which does not
admit a symmetric Higgs field with non-trivial determinant. As we have seen, this holomorphic structure
is isomorphic to $\dbar^0.$

Let $\nabla$ be a flat unitary Lawson symmetric connection such that $(\nabla)''=\dbar^0.$
As we have seen in the proof of Lemma \ref{exceptional_structures}, $\dbar^0$ admits a nowhere vanishing 
symmetric Higgs field $\Psi\in H^0(M,K\End_0(V,\dbar^0))$ with $\det\Psi=0.$ The kernel of $\Psi$ is the dual of
the spin bundle $S$ of the Lawson surface.
We split the connection
 \[\nabla=\dvector{\dbar^{S^*} & \bar q \\ 0 & \dbar ^S}+\dvector{\del^{S^*} & 0 \\ -q & \del ^S}\] 
 with respect to the unitary decomposition $V=S^*\oplus S\to M.$ 
 Note that $q$ is a multiple of the Hopf differential $Q$ of the Lawson surface and that 
 $\bar q\in \Gamma(M,\bar KK^{-1})$ is its adjoint with respect to the unitary metric. As explained above, we want to 
 study $\nabla=\nabla^0$ as a limit
of a family of flat Lawson symmetric connections \[t\mapsto\nabla^t,\] 
such that the holomorphic structures vary non-trivially in $t.$
 We restrict to the case where a choice of a corresponding line bundle
$L_t^+\in  Jac(\tilde M/\Z_3)$ with $\Pi(L_t)=[(\nabla^t)'']$ is given by the holomorphic structure 
\[\dbar_0+t d\bar z,\] 
where $\dbar_0=d''$ is the trivial holomorphic structure on $\underline \C\to \tilde M/\Z_3.$ 
As $\Pi$ branches at $\underline \C$ (Proposition
\ref{PI}) this can always be achieved by 
rescaling the family as long as the map $t\mapsto [(\nabla^t)'']\in\mathcal S$ has a branch point of order $1$ at $0.$
Pulling the family of connections back to $\tilde M$ (and applying gauge transformations to them which depend 
holomorphically on $t$ on a disc containing $t=0$) the holomorphic structures of the connections take the following form
\[(\pi^*\nabla^t)''=\dvector{\dbar^{S^*} +t \bar\eta& \bar q 
\\ 0 & \dbar ^S-t\bar\eta},\]
where $\bar\eta=\tilde\pi^*d\bar z.$ A family of symmetric Higgs fields $\Psi_t\in H^0(M, K_M\End_0(V,(\nabla^t)'')$ is given by
\begin{equation}\label{Higgspansion}
\pi^*\Psi_t=\dvector{t c \eta & \omega+t\beta(t)\\ 0 & -t c \eta}
\end{equation}
after pulling them back as $1$-forms to $\tilde M.$
Here  $\beta(t)$ is a $t$-dependent section of 
$\pi^*K_M=K_{\tilde M}\Hom(\pi^*S,\pi^*S^*),$ and $\omega\in H^0(\tilde M,K_{\tilde M}\Hom(\pi^*S,\pi^*S^*))$ is the canonical section which has
 zeros at the branch points of $\pi$ and $c$ is a some non-zero constant.
Note that $\omega$ can be considered as the pull-back of the bundle-valued 
 $1$-form $1\in H^0(M,K\Hom(S,S^*)),$ or as a square root of $\eta.$ With respect to the fixed (non-holomorphic) background decomposition $\pi^*V=\pi^*S^*\oplus \pi^*S$ the eigenlines 
$L_\pm^t$ of $\pi^*\Psi_t$ 
on $\tilde M$ are
\[\dvector{1\\0}\] and 
\[\dvector{1 \\-2 c\omega t+t^2( ...)}.
\]
Therefore the expansion in $t$ of the singular gauge transformation 
$f_t\colon L_+^t\oplus L_-^t\to \pi^*V=\pi^*S^*\oplus \pi^*S$ is given by
\[\dvector{1 & 1\\0 &-2c\omega t+t^2(...)}.
\]
The expansion of $\pi^*\nabla^t$ is of the form
\[\pi^*\nabla^t=\pi^*\nabla+t\dvector{ \bar\eta & 
0 \\ 0 & -\bar\eta}+t \Gamma(t),
\]
where $\Gamma(t)\in\Gamma(\tilde M,K_{\tilde M}\End_0(V))$ depends holomorphically on $t.$
Applying the gauge $f_t$ we obtain the following asymptotic behavior
\[\nabla^t\cdot f_t=\frac{1}{t}\dvector{-\frac{\pi^*q}{2c\omega} & 0\\ 0 &\frac{\pi^*q}{2c\omega}}+..\, .
\]
The pullback $\pi^*q\in H^0(K_{\tilde M}K_M)$ has zeros of order $3$ at the branch points of $\pi$ and therefore it is a 
constant multiple of $\eta \omega.$ Hence, the holomorphic line bundle connections on $E^t$ given by the $1:1$ 
correspondence in Theorem \ref{abelianization}
have the following expansion
\begin{equation}\label{explicit_expansion_dbar0_tilde}
\nabla^{E^t}=d+td\bar z+\frac{\tilde c}{t}dz+\hat e(t)dz
\end{equation}
for some holomorphic function $\hat e(t).$ In order to determine $\tilde c,$ we expand the
family of equations $(\nabla^t)''\Psi_t=0$ as follows:
\[0=(\pi^*\nabla^t)''\pi^*\Psi_t=t\dvector{0 & 
-2\pi^*\bar q c \eta+2\omega \bar \eta+\dbar^{\pi^*K_M}\beta(0) \\0 & 0}+t^2(...).
\]
As we have fixed $\omega\in H^0(\tilde M,K_{\tilde M}\Hom(S,S^*))=H^0(\tilde M,\pi^*K_M)$ up to sign by
$\omega^2=\eta=\tilde\pi^*dz$ we obtain from Serre duality applied to the bundle $\pi^*K_M$
\begin{equation}\label{24i}
\int_{\tilde M}\pi^*\bar q c \eta \omega=\int_{\tilde M}\bar \eta \omega^2=3 \int_{\tilde M/\Z_3}d\bar z\wedge dz=24i.
\end{equation}
Recall that we have identified $\tilde M/\Z_3\cong \C/(2\Z+2i\Z)$ and $dz$ is the corresponding differential.
The degree of $\pi^*S^*\to\tilde M$ is $-2$ and we obtain from the flatness of $\nabla$ that
\begin{equation}\label{4pii}
4\pi i=\int_{\tilde M}\pi^*\bar q\wedge\pi^*q.
\end{equation}\label{asymptotic_constant}
Combining \eqref{24i} and \eqref{4pii} we obtain
\begin{equation}
-\frac{\pi^*q}{2c\omega}=-\frac{\pi}{12}\eta,
\end{equation}
which exactly tells us the asymptotic of the family \ref{explicit_expansion_dbar0_tilde}.
\begin{The}\label{extension_dbar0}
Let $\nabla^t$ be a holomorphic family of flat Lawson symmetric connections on $M$ such that $(\nabla^0)''$ is 
isomorphic to $\dbar^0.$ If
$t\mapsto[(\nabla^t)'']\in\mathcal S$ branches of order $1$ at $t=0,$ then, after reparametrization the family, 
$\nabla^t$ induces
by means of Theorem \ref{abelianization} and \eqref{connection1form}
a meromorphic family of flat connections of the form 
\begin{equation}\label{explicit_expansion_dbar0}
\tilde\nabla^{t}=d+td\bar z-\frac{\pi}{12t}dz+ t e(t)dz
\end{equation}
on $\underline\C\to \tilde M/\Z_3,$ where $e(t)$ is a holomorphic function in $t.$

Conversely, let $\tilde\nabla^t$
be a meromorphic family of flat connections on 
$\underline\C\to \tilde M/\Z_3$ of the form \ref{explicit_expansion_dbar0}.
Then the induced family of flat Lawson symmetric connections $\nabla^t$  
on the complex rank $2$ bundle $V\to M$ 
extends (after a suitable $t-$dependent gauge) holomorphically to $t=0$
such that  $\nabla^0$ is a flat Lawson symmetric connection and
$(\nabla^0)''$ is isomorphic to $\dbar^0.$
\end{The}
\begin{proof}
Our primarily discussion was restricted to the case where $\nabla^0$ is unitary.
In that case it remains to show that the function $\hat e(t)$ in \eqref{explicit_expansion_dbar0_tilde} 
has a zero at $t=0.$ This follows from the fact that the function $a^u$ in Theorem \ref{unitary_a} is odd.
For the general case we need to study the effect of adding a holomorphic family of Lawson symmetric Higgs fields
\[\Psi(t)\in H^0(M;K\End_0(V,(\nabla^t)'')).\]
Such a holomorphic family of Higgs fields is given by 
\[h(t)\dvector{t c \eta & \omega+t\beta(t)\\ 0 & -t c \eta}\]
for some function $h(t)$ which is holomorphic in $t,$
see  \eqref{Higgspansion}.
From this the first part easily follows. Moreover, by reversing the arguments one also obtains a proof of the converse 
direction.
\end{proof}
\begin{Cor}
The unitarizing function $a^u\colon\C\setminus\frac{1}{2}\Z+\frac{i}{2}\Z\to\C$ in Theorem \ref{unitary_a} is given by
\[a^u(x)=-\frac{1}{12\pi}\frac{\theta'(-2x)}{\theta(-2x)}+\frac{1}{12\pi}\frac{\theta'(2x)}{\theta(2x)}+\frac{1}{3}x+\frac{2}{3}\bar x+b(x),\]
where $\theta$ is the $\Theta$-function as in \eqref{def_theta}, $\theta'$ is its derivative  
and
 $b(x)\colon \C\to\C$ is an odd smooth function which is doubly periodic with respect to
 the lattice $\frac{1}{2}\Z+\frac{i}{2}\Z.$
\end{Cor}
\begin{proof}
The function
$\tilde a\colon\C\setminus\frac{1}{2}\Z+\frac{i}{2}\Z\to\C$ defined by
 \[\tilde a(x)=-\frac{1}{12\pi}\frac{\theta'(-2x)}{\theta(-2x)}+\frac{1}{12\pi}\frac{\theta'(2x)}{\theta(2x)}+\frac{1}{3}x+\frac{2}{3}\bar x\]
is an odd function in $x$ which satisfies the same functional equations (see Theorem \ref{unitary_a}) as $a^u.$
Note that the parametrization of the family of holomorphic rank $1$ structures in Theorem \ref{unitary_a} and in 
Theorem \ref{extension_dbar0} differ by the multiplicative factor $-\pi.$ Therefore, $\tilde a$ has the right asymptotic 
behavior at the lattice points $\frac{1}{2}\Z+\frac{i}{2}\Z.$ So the difference $b=a^u-\tilde a$ is an odd, smooth and 
doubly periodic function.
\end{proof}

\subsection{The case of the non-stable holomorphic structure}
We have already seen in Section \ref{non_semi_stable} that every flat Lawson symmetric connection on $M$
whose holomorphic structure is not semi-stable is gauge equivalent to
\[\nabla=\dvector{&\nabla^{spin*}& 1\\ & \vol+c Q & \nabla^{spin}}\]
with respect to $V=S^*\oplus S\to M.$ 
In this formula $\nabla^{spin}$ and $\vol$ are induced by the Riemannian metric of 
constant curvature $-4,$ $c\in\C$ and $Q$ is the Hopf differential of the Lawson surface. The gauge 
orbit of the holomorphic structure $\nabla''$ is infinitesimal close to the gauge orbits of the holomorphic structures
$\dbar^0$ and $\dbar^S\oplus\dbar^{S^*}.$ As in Section \ref{The case of the stable holomorphic structure},
we approximate $\nabla$ by a holomorphic family of flat Lawson symmetric connections $t\mapsto \nabla^t$
such that the isomorphism classes of the holomorphic structures $(\nabla^t)''$ vary in $t.$
We obtain a similar result as Theorem \ref{extension_dbar0}.
\begin{The}\label{extension_non_stable}
Let $\nabla^t$ be a holomorphic family of flat Lawson symmetric connections on $M$ such that
 $(\nabla^0)''$ is isomorphic to the non-trivial 
extension $S\to V\to S^*$ and such that 
$t\mapsto[(\nabla^t)'']\in\mathcal S$ branches of order $1$ at $t=0.$ After reparametrization the family, 
$\nabla^t$ corresponds (via Theorem \ref{abelianization} and \eqref{connection1form})
to a meromorphic family of flat connections $\tilde\nabla^t$ on $\underline\C\to \tilde M/\Z_3$ of the form 
\begin{equation}\label{explicit_expansion_dbar0_uniform}
\tilde\nabla^{t}=d+td\bar z+\frac{\pi}{12t}dz+ t e(t)dz,
\end{equation}
where $e(t)$ is holomorphic in $t.$

Conversely, let $\tilde\nabla^t$
be a meromorphic family of flat connections on 
$\underline\C\to \tilde M/\Z_3$ of the form \ref{explicit_expansion_dbar0_uniform}.
Then the induced family of flat Lawson symmetric connections $\nabla^t$  
on the complex rank $2$ bundle $V\to M$ 
extends (after a suitable $t-$dependent gauge) holomorphically to $t=0$
such that $(\nabla^0)''$ is isomorphic to the non-trivial extension $0\to S\to V\to S^*\to 0.$ Moreover,
$\nabla^0$ is gauge equivalent to the uniformization connection (see \eqref{uniformization_connection}) if the 
function $e$ has a zero at $t=0.$
\end{The}
\begin{proof}
Consider a holomorphic family of flat Lawson symmetric connections $\hat\nabla^t$ such that
$(\nabla^t)''$ is isomorphic to $(\hat\nabla^t)''$ for all $t$ and such that $\hat\nabla^0$ is unitary. In particular, 
$(\hat\nabla^0)''$ is isomorphic to $\dbar^0.$ Then, after applying the $t$-dependent gauge $g_t$
 the difference \[\Psi_t:=\hat\nabla^t-g_t^{-1}\nabla^tg\in H^0(M,K\End_0(V,(\hat\nabla^t)''))\]
satisfies 
\[\det\Psi_t=\frac{q}{t}+\,\text{higher order terms},\]
where $q$ is a non-zero multiple of the Hopf differential. 
This implies, that the line bundle connections
$\tilde\nabla^t$ have an expansion like
\[\tilde\nabla^{t}=d+td\bar z+\frac{c}{t}dz+\,\text{higher order terms}\]
for some non-zero $c\in\C.$
Then, analogous to the computation in Section \ref{The case of the stable holomorphic structure},
 one obtains $c=\frac{1}{12\pi}.$ Note that the reason for the different signs is because of
 the last sign in the degree formula for $S^*:$
 \[-2\pi i\deg(S^*)=\int_M\bar q\wedge q=-\int_M 1\wedge vol.\]
 To show that the $0.$-order term in the expansion of $\tilde\nabla^t$ vanishes we first observe that there exists
 an additional (holomorphic) symmetry $\tilde\tau\colon M\to M$ which induces the symmetry $z\mapsto iz$ on
 $\tilde M/\Z_3.$ Note that $\tilde\tau^*Q=-Q.$ Because the gauge equivalence class of the uniformization connection 
 (\eqref{uniformization_connection}) is also invariant under $\tilde\tau,$ one easily gets (as in the proof of
 Theorem \ref{explicit_expansion_dbar0_tilde}) that the $0.$-order term vanishes. Moreover one obtains that in the case 
 of the uniformization connection also the first order term vanishes.
 \end{proof}

\section{The spectral data}\label{sec:the_spectral_data}
So far we have seen that the generic Lawson symmetric flat connection is determined (up to gauge equivalence), after 
the choice of one eigenline bundle of a symmetric Higgs field, by a flat line bundle connection on a square torus. 
Moreover, 
the remaining flat connections are explicitly given as limiting cases of the above construction. We now apply
these results to the case of the family of flat connections $\nabla^\lambda$ associated to a minimal
 surface. We assume that the minimal surface is of genus $2$ and has the conformal type and the symmetries 
 $\varphi_2,$ $\varphi_3$ and $\tau$ of the Lawson surface. The family of flat connections induces a family of 
 Lawson symmetric holomorphic structures $\dbar^\lambda=(\nabla^\lambda)''$ which extends to $\lambda=0.$
 As it is impossible to make a consistent choice of the eigenline bundles of symmetric Higgs fields
 with respect to $\dbar^\lambda$ for all $\lambda\in \C^*$ (see Proposition \ref{The definition of the spectral curve}) we need to 
 introduce a so-called spectral curve which double covers the
 spectral plane $\C^*$ and enables us to parametrize the eigenline bundles. Then, the family of flat connections 
 $\nabla^\lambda$ is determined (up to a $\lambda$-dependent gauge) by the corresponding family of flat line 
 bundles over the torus. The behavior of this family of flat line bundles is very similar (at least around $\lambda=0$)
 to the family of flat line bundles parametrized by the spectral curve of a minimal or CMC torus, compare with \cite{H}. 
 The main difference is that we have some kind of symmetry breaking between $\lambda=0$ and $\lambda=\infty:$
  We do not treat the holomorphic and 
 anti-holomorphic structures of a flat connection in the same way but consider the moduli space of flat connections as
 an affine bundle over the moduli space of holomorphic structures. As a consequence, we do not have an explicitly 
 known reality condition, which seems to be the missing ingredient to explicitly determine the Lawson surface. 
 
  By taking the gauge equivalence classes of the associated family of holomorphic structures $\dbar^\lambda$
   we obtain a holomorphic map
 \[\mathcal H\colon \C\to \mathcal S\cong\P^1\]
 to the moduli space of semi-stable Lawson symmetric holomorphic structures, see Proposition \ref{projective_line}.
 This map is given by $\mathcal H(\lambda)=[\dbar^\lambda]$ for those $\lambda$ where $\dbar^\lambda$ is 
 semi-stable. By remark \ref{two_points} it extends holomorphically to the points $\lambda$ where $\dbar^\lambda$ is not semi-stable.
 \begin{Pro}[The definition of the spectral curve]\label{The definition of the spectral curve}
 There exists a holomorphic double covering $p\colon\Sigma\to \C$ defined on a Riemann surface $\Sigma,$
 the spectral curve,
  together with a holomorphic map $\mathcal L\colon\Sigma\to Jac(\tilde M/\Z_3)$  such that 
 
 \begin{xy}
\hspace{5.5 cm}\xymatrix{
      \Sigma \ar[rr]^{\mathcal L} \ar[d]^{p}  &  &   Jac(\tilde M/\Z_3) \ar[d] ^{\Pi} \\
      \C \ar[rr]_{\mathcal H}    &       &   \mathcal S   
  }
\end{xy}
 
commutes, where $\Pi\colon Jac(\tilde M/\Z_3)\to\mathcal S$ is as in Proposition \ref{PI}. The map $p$
 branches over $0\in\C.$

 \end{Pro}
\begin{proof}
We first define 
\begin{equation*}
\begin{split}
\underline\Sigma
&=\{(\lambda, L)\in\C\times Jac(\tilde M/\Z_3)\mid \Pi(L)=\mathcal H(\lambda)\}
\end{split}
\end{equation*}
which is clearly a non-empty analytic subset of $\C\times Jac(\tilde M/\Z_3).$ 
Then, the spectral curve is given by the normalization
\[\Sigma\to \underline\Sigma,\]
and $\mathcal L$ is the composition of the normalization with the projection onto the second factor.
It remains to prove that $\Sigma$ branches over $0.$ Because $\Pi$ branches over $[\dbar^0]$
this follows if we can show that the map $\mathcal H$ is immersed at $\lambda=0.$
 As $\dbar^0$ is stable, the tangent space at $[\dbar^0]$ of the moduli space of (stable) holomorphic 
structures with trivial determinant is given by $H^1(M,K\End_0(V,\dbar^0)).$ The cotangent space is 
given via trace and integration by 
$H^0(M,KEnd_0(V,\dbar^0)).$ With
\[\frac{\partial }{\partial\lambda}\dbar^\lambda=\dvector{0 & 0\\ \vol &0}\]
and \[\Phi=\dvector{0 & 1\\ 0 &0}\in H^0(M, K\End_0(V,\dbar^0))\] we see that 
\[\Phi(\frac{\partial }{\partial\lambda}\dbar^\lambda_{\mid \lambda=0})=\int_M\vol\neq0\]
which implies that $\mathcal H$ is immersed at $\lambda=0.$
\end{proof}

In order to study the family of gauge equivalence classes $[\nabla^\lambda]$ we
 consider the moduli space of flat $\C^*$-connection on $\tilde M/\Z_3$ as an affine holomorphic bundle
\[\mathcal A^f\to Jac(\tilde M/\Z_3)\]
over the Jacobian by taking the complex anti-linear part of a connection.
Then, as a consequence of Theorem \ref{abelianization} together with our discussion in 
Section \ref{exceptional_connections},
we obtain a meromorphic lift

 \begin{xy}
\hspace{5.5 cm}  \xymatrix{
        &   &       \mathcal A^f \ar[d]^{''}  \\
                \Sigma \ar[rru]^{\mathcal D} \ar[rr]_{\mathcal L}       &   &  Jac(\tilde M/\Z_3)
  }
\end{xy}

of the map $\mathcal L$ which parametrizes the gauge equivalence classes  $[\nabla^\lambda].$
The map $\mathcal D$ has poles at those points $p\in\Sigma$ where
$\mathcal L(p)=\underline\C\in Jac(\tilde M/\Z_3)$ is the trivial holomorphic bundle. 
Note that the (unique) preimage $0\in\Sigma$ of $\lambda=0$ always satisfies $\mathcal L(p)=\underline\C.$
The poles at $p\neq 0$ are generically simple, and the exact asymptotic
behavior of $\mathcal D$ around $p$ is determined by the results of Section \ref{exceptional_connections}.
\begin{Def}
The triple $(\Sigma,\mathcal L,\mathcal D),$ which is determined by the associated family of flat connections of a 
compact minimal surface in $S^3$
with the symmetries $\varphi_2,$ $\varphi_3$ and $\tau$ of the Lawson surface of genus $2,$ is called spectral data
of the surface.
\end{Def}
\subsection{Asymptotic behavior of the family of flat connections}
We have already seen that the spectral curve $\Sigma$ of a compact minimal surface in $S^3$
with the symmetries of the Lawson surface branches over $\lambda=0$ and that the map $\mathcal L$ is holomorphic.
We claim that the asymptotic behavior of $\mathcal D$ around the preimage of $\lambda=0$
is analogous to the case of minimal tori in $S^3$ \cite{H}.

In order to show this we consider a holomorphic family of flat Lawson symmetric $\SL(2,\C)$-connections
\[\lambda\mapsto\hat\nabla^\lambda\] defined on an open neighborhood of $\lambda=0$ such that 
$(\hat\nabla^\lambda)''=\dbar^\lambda.$ This implies that for small $\lambda\neq0$ the difference 
\[\nabla^\lambda-\hat\nabla^\lambda\]
is a symmetric Higgs field $\Psi\in H^0(M,K\End_0(V,\dbar^\lambda))$ whose determinant is a non-zero 
multiple of $Q.$
An expansion of $\hat\nabla^\lambda$ around $\lambda=0$ is given by

\begin{equation}\label{nablahat}
\hat\nabla^\lambda=\dvector{&\nabla^{spin^*} +\omega_0& -\frac{i}{2} Q^*+\alpha \\ &-\frac{i}{2}q & \nabla^{spin}-\omega_0}
+\lambda\dvector{&\omega_1& \alpha_1 \\ & -\vol+\beta_1 & -\omega_1}+...,
\end{equation}
where $q$ is a constant multiple of the Hopf differential, $\alpha_i\in\Gamma(M,KK^{-1}),$ 
$\omega_i\in\Gamma(M,K)$ and $\beta_1\in\Gamma(M,K^2).$ We claim that 
\[Q-q\neq 0\]
is a non-zero constant multiple of the Hopf differential.
To see this note that 
\begin{equation*}
\begin{split}
F^{\nabla^{spin^*}}&=\frac{1}{4}Q^*\wedge Q+\Tr(\Phi\wedge\Phi^*)\\
&=\frac{1}{4}Q^*\wedge q-\dbar\omega_0
\end{split}
\end{equation*}
 as a consequence of the flatness of $\nabla^\lambda$ as well as of $\hat\nabla^0.$
 The claim then follows from $\int_M\dbar\omega_0=0$ and $\int_M\Tr(\Phi\wedge\Phi^*)\neq0.$
Comparing \eqref{nablahat} with Proposition \ref{connection_data} in appendix \ref{A1} we obtain
\begin{equation}\label{asymptotic_determinant}
\det(\nabla^\lambda-\hat\nabla^\lambda)=-\frac{i}{4}\lambda^{-1}(Q-q)+...\,.
\end{equation}
This leads to the following theorem. 
\begin{The}\label{asymptotic_at_0}
Let $(\Sigma,\mathcal L,\mathcal D)$ be the spectral data associated to a compact minimal surface in $S^3$
with the symmetries $\varphi_2,$ $\varphi_3$ and $\tau$ of the Lawson surface of genus $2.$
Let $t$ be a coordinate of $\Sigma$ around $p^{-1}(\{0\})$ such that locally
\[\mathcal L(t)=\dbar_0+td\bar z,\]
 where $z$ is the affine coordinate on $\tilde M/\Z_3\cong\C/(2\Z+2i\Z).$
The asymptotic of the map $\mathcal D$ at $0$ is given by
\[\mathcal D(t)=d+t \bar dz+(c_{-1}\frac{1}{t}+c_1 t+..) dz\]
for some $c_{-1}\neq\pm\frac{\pi}{12}$ and with respect to the natural local trivialization of the affine bundle 
$\mathcal A^f\to Jac(\tilde M/\Z_3).$

The covering $p\colon\Sigma\to\C$ branches at most over those points $\lambda\in\C$ where $\dbar^\lambda$
is one of the exceptional holomorphic structures, i.e., $\mathcal L(\mu)=\underline\C$ for $p(\mu)=\lambda.$
Moreover
$\mathcal D$ satisfies the reality condition 
\[\mathcal D(\mu)=\mathcal U(\mathcal L(\mu))\]
for all $\mu\in p^{-1}(S^1)\subset\Sigma$
where $\mathcal U$ is the section given by Theorem \ref{unitary_a}, and the closing condition
\[\mathcal D(\mu)=[d+\frac{-1+i}{4}\pi dz+\frac{1+i}{4}\pi d\bar z]\]
for all $\mu\in p^{-1}(\{\pm 1\})\subset\Sigma.$
\end{The}
\begin{proof}
As in the proof of Theorem \ref{extension_dbar0} we see that the effect of adding a family of Higgs fields
$\nabla^\lambda-\hat\nabla^\lambda$ with asymptotic as in \eqref{asymptotic_determinant}
on the corresponding $\C^*$-connections over $\tilde M/\Z_3$ is given by adding 
\[(\frac{c_{-1}}{t}+c_0+c_1 t+...)dz \]
with $0\neq c_1,c_0,c_1\in\C.$
As $\det(\nabla^\lambda-\hat\nabla^\lambda)$ is even in $t$ by the definition of $t,$
the constant $c_0$ vanishes. Together with Theorem \ref{extension_dbar0} this implies the first statement.

The reality condition is a consequence of the fact that the connections $\nabla^\lambda$ are unitary for 
$\lambda\in S^1$ and of Theorem \ref{unitary_a}. The closing condition follows from the observation that the trivial 
connection of rank $2$ on $M$ corresponds to the connection \[d+\frac{-1+i}{4}\pi dz+\frac{1+i}{4}\pi d\bar z\] on 
$\tilde M/\Z_3.$

It remains to prove that the spectral curve cannot branch over the points $\lambda\in\C$ where $\dbar^\lambda$ is
semi-stable and not stable. For this we consider the holomorphic line bundle $L\to\C$ whose fiber is at a generic point
$\lambda$ 
spanned by the 1-dimensional space symmetric Higgs fields of $\dbar^\lambda.$ But the space of symmetric Higgs 
fields at the semi-stable points is also 1-dimensional, and the determinant of a non-zero symmetric Higgs field is a 
non-zero multiple of the Hopf differential $Q.$ Therefore, the eigenlines of the Higgs fields  can be parametrized
in $\lambda$ as long as $\dbar^\lambda$ is not an exceptional holomorphic structure.

\end{proof}

\subsection{Reconstruction}
Conversely, a hyper-elliptic Riemann surface $\Sigma\to \C$ 
together with a map $\mathcal L\colon \Sigma\to Jac(\tilde M/\Z_3)$ and a lift $\mathcal D$ into
the affine bundle of line bundle connections which
satisfy the asymptotic condition, the reality and closing conditions of Theorem \ref{asymptotic_at_0} give rise
to a compact minimal surface of genus $2$ in $S^3.$ To prove this we first need some preparation:
\begin{The}\label{general_reconstruction}
Let $\lambda\in\C^*\mapsto\tilde\nabla^\lambda$ be a holomorphic family of flat $\SL(2,\C)$-connections
on a rank $2$ bundle $V\to M$ over a compact Riemann surface $M$ of genus $g\geq2$ such that
\begin{itemize}
\item the asymptotic at $\lambda=0$ is given by
\[\tilde\nabla^\lambda\sim \lambda^{-1}\Psi+\tilde\nabla+...\]
where $\Psi\in\Gamma(M,K\End_0(V))$ is nowhere vanishing and nilpotent;
\item for all $\lambda\in S^1\subset\C$ there is a hermitian metric  on $V$ such that $\tilde\nabla^\lambda$ is unitary
with respect to this metric;
\item $\tilde\nabla^\lambda$ is trivial for $\lambda=\pm1.$
\end{itemize}
Then there exists a unique (up to spherical isometries) minimal surface $f\colon M\to S^3$ 
 such that its associated family of
flat connections $\nabla^\lambda$ and the family $\tilde\nabla^\lambda$ are gauge equivalent, i.e., there exists a $\lambda$-dependent holomorphic family of gauge 
transformations $g$ which extends through $\lambda=0$ such that
$\nabla^\lambda\cdot g=\tilde\nabla^\lambda.$
\end{The}
\begin{proof}
 It is a consequence of the asymptotic of $\tilde\nabla^\lambda$  that $(\tilde\nabla^\lambda)''$ is 
stable for generic $\lambda\in\C^*,$ for more details see \cite{He1}. This implies that the generic connection $\tilde\nabla^\lambda$ is irreducible.
Therefore the hermitian metric for which $\tilde\nabla^\lambda$ is 
unitary is unique up to constant multiples for generic $\lambda\in S^1\subset \C^*.$ For those $\lambda\in S^1$ the hermitian metric $(,)^\lambda$ is unique if we impose that it is 
compatible with the determinant on $V,$ i.e., the determinant of an orthonormal basis is unimodular.
The metric $(,)^\lambda$ depends real-analytically on $\lambda\in S^1\setminus S,$ where $S\subset S^1$ is the set of 
points where $\nabla^\lambda$ is not irreducible, and can be extended through the set $S.$

From now on we identify $V=M\times\C^2$ and fix a unitary metric $(,)$ on it. Therefore, $(,)^\lambda$ can be identified
with a section $[h]\in\Gamma(S^1\times M,\SL(2,C)/\SU(2))$ which itself can be lifted to a section
$h\in\Gamma(S^1\times M,\SL(2,C)).$ Clearly, $h$ is real analytic in $\lambda$ and satisfies
\[h_\lambda^*\, (,)=(,)^\lambda.\]
We now apply the loop group Iwasawa decomposition to $g=h^{-1},$ i.e.,
\[g=B F,\]
where $B\in\Gamma(D^1\times M,\SL(2,\C))$ is holomorphic in $\lambda$ on 
$D^1=\{\lambda\in\C\mid \bar\lambda\lambda\leq 1\} $ and $F\in\Gamma(S^1\times M,\SU(2))$ is unitary, see \cite{PS}
for details. Gauging
\[\nabla^\lambda=\tilde\nabla^\lambda\cdot B\]
we obtain a holomorphic family of flat connections $\nabla^\lambda$ on $D^1\setminus\{0\}$ which is unitary with 
respect to $(,)$ on $S^1$ by construction. Applying the Schwarz reflection principle yields a holomorphic family 
of flat connection $\lambda\in\C^*\mapsto\nabla^\lambda$ which is unitary on $S^1$ and trivial for $\lambda=\pm1.$
Moreover, as $B$ extends holomorphically to $\lambda=0,$ $\nabla^\lambda$ has the following asymptotic
\[\nabla^\lambda\sim \lambda^{-1}\Phi+\nabla+..\]
where $\Phi=B_0^{-1}\Psi B_0$ is complex linear, nowhere vanishing and nilpotent. Using the Schwarzian reflection
again, we obtain 
\[\nabla^\lambda=\lambda^{-1}\Phi+\nabla-\lambda\Phi^*\] for a unitary connection $\nabla.$ This proves the existence 
of an associated minimal surface $f\colon M\to S^3.$  

Let $f_1, f_2\colon M\to S^3$ be two minimal surface such that their associated families of
flat connections $\nabla_1^\lambda$ and $\nabla_2^\lambda$   are gauge equivalent to $\tilde\nabla^\lambda,$
where both families of gauge transformations extend holomorphically to $\lambda=0.$
Let $g\in\Gamma(\C\times M,\SL(2,\C))$ be the gauge between these two families which, by assumption, also extends 
to $\lambda=0.$
We may assume that for all $\lambda\in S^1$ the connections $\nabla_1^\lambda$ and $\nabla_2^\lambda$  are unitary 
with respect to 
the same hermitian metric. As the connections are generically irreducible the gauge $g$ is unitary along the unit circle.
  By the Schwarz
 reflection principle $g$ extends to $\lambda=\infty,$ and therefore $g$ is constant in $\lambda.$ Hence, the 
 corresponding minimal surfaces $f_1$ and $f_2$ are the same up to spherical isometries.
\end{proof}
\begin{Rem*}
There exists similar results as Theorem \ref{general_reconstruction} and Theorem 
\ref{reconstruction_spec} for the DPW approach to minimal surface, see 
\cite{SKKR} and \cite{DW}.
\end{Rem*}
\begin{Rem}\label{smoothness}
The above theorem is still true if the individual connections $\tilde\nabla^\lambda$ are only of class  $C^k$ for
$k\geq3$ and not necessarily
smooth.
\end{Rem}
Similar to the case of tori, the knowledge of the gauge equivalence class of the
associated family of flat connections $[\nabla^\lambda]$ for all $\lambda$ is
in general not enough to determine the minimal immersion uniquely. 
The freedom is given by $\lambda$-dependent meromorphic gauge transformations $g$
which is unitary along the unit circle. Applying such
a gauge transformation is known in the literature as dressing, see for example \cite{BDLQ} or \cite{TU}.
For tori, dressing is closely related to the isospectral deformations induced by changing the eigenline bundle of a
 minimal immersion. 
 In fact, simple factor dressing with respect to special eigenlines of the connections $\nabla^\lambda$ (those which 
 correspond to the eigenline bundle) generate the abelian group of isospectral deformations.
 The remaining eigenlines, which only occur at values of $\lambda$ where the monodromy takes values in
  $\{\pm \Id\}$, produce
 singularities in the spectral curve and therefore do not correspond to isospectral deformations in the sense of Hitchin.
 Due to the fact that for minimal surfaces of higher genus the generic connection $\nabla^\lambda$ is irreducible
 there are in general no continuous families of dressing deformations:
 \begin{The}\label{dressing}
 Let $f,\tilde f\colon M\to S^3$ be two conformal minimal immersions from a compact Riemann surface of genus
  $g\geq2$ together 
 with their associated families of flat connections $\nabla^\lambda$ and $\hat\nabla^\lambda.$ Assume that 
 $\nabla^\lambda$ is gauge equivalent to $\tilde\nabla^\lambda$ for generic $\lambda\in\C^*.$
 Then there exists a meromorphic map
 \[g\colon\CP^1\to \Gamma(M,\End(V))\] 
 such that $\nabla^\lambda\cdot g=\tilde\nabla^\lambda.$
 This map $g$ is holomorphic and takes values in the invertible endomorphisms
  away from those $\lambda_0\in\C^*$ where $\nabla^{\lambda_0}$ or equivalently
   $\tilde\nabla^{\lambda_0}$ is reducible. 
   
   The space of such dressing deformations of surfaces $f\mapsto\tilde f$ is generated by simple factor dressing, i.e., by
    maps $d\colon\CP^1\to \Gamma(M,\End(V))$ of the form
   \[ d(\lambda)=\pi^L+\frac{1-\bar\lambda_0^{-1}}{1-\lambda_0}\frac{\lambda-\lambda_0}{\lambda-\bar\lambda_0^{-1}}\pi^{L^\perp},\] 
   where $L$ is an eigenline bundle of the connection $\nabla^{\lambda_0}$ and $L^\perp$ is its orthogonal complement.
  \end{The} 
 \begin{proof}
 We first show that $\nabla^\lambda$ and $\tilde\nabla^\lambda$ are gauge equivalent away from those 
 $\lambda_0\in\C^*$ where $\nabla^{\lambda_0}$ or $\tilde\nabla^{\lambda_0}$ is reducible.
 The gauge between two irreducible gauge equivalent connections $\nabla^\lambda$ and $\tilde\nabla^\lambda$ is  
 unique up to a constant multiple of the identity. Moreover, multiples of this gauge are the only parallel endomorphisms
 with respect to the connection $\nabla^\lambda\otimes(\tilde\nabla^\lambda)^*.$ As the connections depend holomorphic on
 $\lambda$ there exists a holomorphic line bundle $\mathcal G\to \C^*$ whose line at $\lambda\in\C^*$ is a subset
of the parallel endomorphisms (and coincides with it whenever $\nabla^\lambda$ or equivalently $\tilde\nabla^\lambda$ 
is irreducible). A non-vanishing section $g\in\Gamma(U,\mathcal G)$ around $\lambda\in U\subset \C^*$ 
gives rise to the gauge between $\nabla^\lambda$ and $\tilde\nabla^\lambda$ as long as $g^\lambda$ is an
isomorphism. This can fail only in the case where $\nabla^\lambda$ or equivalently $\tilde\nabla^\lambda$ is reducible.
 We need to prove that $g$ extends holomorphically to an isomorphism at $\lambda=0.$ Then, as
 $g$ is unitary along the unit circle, $g$ also extends holomorphically to an isomorphism at $\lambda=\infty$ by the 
 Schwarz reflection principle. Note that locally around $\lambda=0$ all connections are irreducible and all
 holomorphic structures are stable away from $\lambda=0.$ Then, as above, there exists a family of gauge
 transformations
  $g^\lambda$ which extend to a holomorphic endomorphism $g^0$ with respect to 
  $\dbar^0\otimes(\tilde\dbar^0)^*.$ From the fact that the connections around $\lambda=0$ are gauge 
  equivalent and the expansions of the two families one deduces that $g^0$ is a holomorphic endomorphism
  between the stable pairs $(\tilde\dbar,\tilde\Phi)$ and $(\dbar,\Phi),$ i.e., $\Phi\circ g^0=g^0\circ\tilde\Phi.$
  Therefore Proposition (3.15) of \cite{H1} implies that $g^0$ is an isomorphism.
  
  In order to find the globally defined dressing $g\colon\CP^1\to \Gamma(M,\End(V))$ we first investigate
  the bundle $\mathcal G\to\C^*.$ We have seen that it extends to $\lambda=0$ holomorphically
  and by switching to anti-holomorphic structures, one can also show that it extends holomorphically to 
  $\lambda=\infty.$ Therefore there exists a meromorphic section $\tilde g\in \mathcal M(\CP^1,\mathcal G)$ whose
  only poles are at $\lambda=\infty.$ As $\mathcal G$ is a holomorphic subbundle of 
  $\C\times\Gamma(M;\End(V))$ the determinant is a holomorphic map
  \[\det\colon\mathcal G\to\C.\]
  Consider the holomorphic function
  \[h\colon\C\to \C,\, h(\lambda)=\det(\tilde g^\lambda).\]
  Note that we may assume that $h$ is non-vanishing along the unit circle as $\tilde g$ is a complex multiple
  of a unitary gauge there.
 The Iwasawa decomposition 
 $h=h_+ h_u$ 
determines functions 
$h_+\colon \{\lambda\mid \lambda\bar \lambda\leq 1\}\to\C^*$ and $h_u\colon \C^*\to \C$ which satisfies
 $\parallel h_u(\lambda)\parallel=1$ for $\lambda\in S^1.$ These are unique up to unimodular constants.
 The square root $\sqrt{h_+}$ is then well-defined on $\{\lambda\mid \lambda\bar \lambda\leq 1\}$
 and we define
 \[g=\frac{1}{\sqrt{h_+}}\tilde g\in H^0(\{\lambda\mid \lambda\bar \lambda\leq 1\},\mathcal G).\]
 The determinant $\det g$ is unimodular along the unit circle, and therefore, $g$ is unitary along the unit circle.
 By the Schwarz reflection principle, we obtain a meromorphic map
 $g\in\mathcal M(\CP^1,\mathcal G)$ which satisfies $\nabla^\lambda\cdot g=\tilde\nabla^\lambda$ by construction.
 
 It is shown in \cite{BDLQ} that  a simple factor dressing $\nabla^\lambda\mapsto\nabla^\lambda\cdot d$ makes
the associated family of a new minimal surface. We want to show by induction that any $g$ as above is the product
of simple factor dressings. Note that $\det g\colon\CP^1\to\CP^1$ is a rational function. If its degree is $0,$
then $\det g$ is a non-zero constant, and $g$ is constant in $\lambda.$ As it is unitary on the unit circle,
$g$ acts as a spherical isometry on the surface. 
Assume that $\det g$ has a zero at $\lambda_0.$ 
As we have seen
$\lambda_0\in\C^*\setminus S^1.$
By multiplying with (a power of) 
$\frac{1-\bar\lambda_0^{-1}}{1-\lambda_0}\frac{\lambda-\lambda_0}{\lambda-\bar\lambda_0^{-1}}\Id$
we can also assume that $g^{\lambda_0}\neq0.$
 As $g^{\lambda_0}$ is a non-zero parallel endomorphism with respect to
$\nabla^{\lambda_0}\otimes(\tilde\nabla^{\lambda_0})$ we see that the line bundle
$L\to M$, which is given by $L_p=\ker g^{\lambda_0}_p$  at generic points $p\in M$,  is an eigenline bundle of 
$\tilde\nabla^{\lambda_0}.$ As a consequence of the unitarity of $\tilde\nabla^\lambda$ along the unit circle, $L^\perp$
is an eigenline bundle of $\tilde\nabla^{\bar\lambda^{-1}}.$ We can apply the simple factor dressing
\[d(\lambda)=\pi^{L^\perp}+\frac{1-\lambda_0}{1-\bar\lambda^{-1}_0}\frac{\lambda-\bar\lambda_0^{-1}}{\lambda-\lambda_0}\pi^L\]
to $\tilde\nabla^\lambda.$ Then, the product $g d$ is again a meromorphic family of gauge
transformations which extends holomorphically through $\lambda_0,$ and the degree of the rational function $\det(gd)$ 
is the degree of the rational function $\det(g)$
minus $1.$
  \end{proof}
 \begin{Lem}\label{breaking_symmetries}
Let $f\colon M\to S^3$ be an immersed minimal surface of genus $2$ having  the symmetries 
 $\varphi_2,$ $\varphi_3$ and $\tau.$ 
 Then a (non-trivial) dressing transformation of such a minimal immersion
 does not admit all symmetries $\varphi_2,$ $\varphi_3$ and $\tau.$
 \end{Lem}
 \begin{proof}
 We have seen in the proof of Theorem \ref{asymptotic_at_0} that there exists a local holomorphic family of Higgs fields
 $\Psi^\lambda\in H^0(M,K\End_0(V,\dbar^\lambda))$ around every point $\lambda_0$ where $\nabla^{\lambda_0}$
 is reducible such that $\det\Psi^\lambda$ is nowhere vanishing. Let $g\colon\CP^1\to \Gamma(M,\End(V))$
 be a meromorphic family of gauges as in Theorem \ref{dressing}. Assume that $g(\lambda_0)$ exists but is a non-zero
 endomorphism which is not invertible.
  It is easy to see that the family of Higgs fields \[\hat\Psi^\lambda=g(\lambda)^{-1}\Psi^\lambda g(\lambda)\]
 with respect to $\tilde\nabla^\lambda=\nabla^\lambda\cdot g$ has a pole at $\lambda_0.$ Then, by resolving the pole by 
 multiplying with an appropriate power of $(\lambda-\lambda_0),$ the (local) holomorphic nowhere vanishing family
 of Higgs fields
 \[\tilde\Psi^\lambda=(\lambda-\lambda_0)^k \hat\Psi^\lambda\]
 satisfies $\det(\tilde\Psi^{\lambda_0})=0.$ 
 This is not possible for a surface $\tilde f$ which has all three symmetries $\varphi_2,$ $\varphi_3$ and $\tau.$
 \end{proof}
\begin{The}\label{reconstruction_spec}
Let $\Sigma$ be a Riemann surface and $p\colon \Sigma\to\C$ be a double covering induced by
the involution $\sigma\colon\Sigma\to\Sigma$ such that $p$ branches over $0.$
 Let $\mathcal L\colon\Sigma\to Jac(\tilde M/\Z_3)$ be a non-constant 
holomorphic map which is odd with respect to $\sigma$
and satisfies $\mathcal L(0)=\underline \C\in Jac(\tilde M/\Z_3).$
Let $\mathcal D\colon\Sigma\setminus p^{-1}(0)\to\mathcal A^f$ be a meromorphic lift of $\mathcal L$ to the moduli space of flat $\C^*$-
connections on $\tilde M/\Z_3$ which is odd with respect to $\sigma$
and which satisfies the conditions of Theorem \ref{extension_dbar0} and Theorem \ref{extension_non_stable} at its 
poles, i.e., $\mathcal D$ defines a holomorphic map from $\C^*$ to the moduli space of flat $\SL(2,\C)$-connections on
$M.$  If $\mathcal D$ has a first order pole at $0$ and satisfies the reality condition 
\[\mathcal D(\mu)=\mathcal U(\mathcal L(\mu))\]
for all $\mu\in p^{-1}(S^1)\subset\Sigma,$
where $\mathcal U$ is the section given by Theorem \ref{unitary_a}, and the closing condition
\[\mathcal D(\mu)=[d+\frac{-1+i}{4}\pi dz+\frac{1+i}{4}\pi d\bar z]\]
for all $\mu\in p^{-1}(\{\pm 1\})\subset\Sigma$ then there exists an immersed minimal surface
$f\colon M\to S^3$ such that $(\Sigma,\mathcal L,\mathcal D)$ are the spectral data of $f.$
Let $t$ be a holomorphic coordinate of $\Sigma$ around $p^{-1}(0)$
such that $t^2=\lambda,$ and consider the expansion
\[\mathcal D\sim d- (x_1 t+..)\pi d\bar z+(a_{-1} \frac{1}{t}+....) \pi dz.\]
Then the area of $f$ is given by
\[\text{Area}(f)=-12 \pi(\frac{1}{6}-2 \pi x_1 a_{-1}).\]

If $p$ only branches at those $\mu\in\Sigma$ where $\mathcal L(\mu)=\underline\C\in Jac(\tilde M/ \Z_3)$
then there is a unique $f$ which has the symmetries $\varphi_2,$ $\varphi_3$ and $\tau.$
\end{The}
\begin{proof}
We first show that the spectral data give rise to a 
holomorphic $\C^*$-family of flat $\SL(2,\C)$-connections on $M$ satisfying the conditions of Theorem
\ref{general_reconstruction}. By assumption we obtain a holomorphic map into the moduli space of flat 
$\SL(2,\C)$-connections on $M.$ Reversing the arguments of Section \ref{sec:flat_connections}
 and \ref{exceptional_connections} we obtain locally on open subsets of $\C^*$ holomorphic families of flat 
 $\SL(2,\C)$-connections on 
 $M$ which are lifts of the map to the moduli space. We cover $\C^*$ by these open sets $U_i,\, i\in \N,$ 
 such that for every $U_i$ there 
 exists at most one point where the corresponding connection is reducible.
 There also exist an open set $U_0$ containing $0$ such that on $U_0\setminus\{0\}$ there exists a lift 
 $\nabla^\lambda_0$ of the map to the 
 moduli space of flat $\SL(2,\C)$-connections on $M$ which has at most a first order pole at $\lambda=0.$  Moreover,
 as $\mathcal L(0)=\underline\C,$ the residuum at $0$ must be a complex linear $1$-form 
 $\Psi\in\Gamma(M,K\End_0(V))$ which is nilpotent. 
We now fix such families  of flat $\SL(2,\C)$-connections $\nabla_i^\lambda$ on every set $U_i.$
Let $\mathcal G$ be the complex Banach Lie group of $C^k$ gauges 
\[\mathcal G=\{g\colon M\to \SL(2,\C)\mid g \text{ is of class } C^k\},\]
where we have fixed a trivialization of the rank $2$ bundle $V=M \times \C^2$ and $k\geq4.$
On $U_i\cap U_j$ we define a map $g_{i,j}\colon U_i\cap U_j\to \mathcal G$ by
\[\nabla^\lambda_j=\nabla^\lambda_i.g_{i,j}.\]
Clearly, the maps $g_{i,j}$ are well-defined, and give rise to a $1$-cocycle of $\C=\cup_{i\in\N_0} U_i$
with values in $\mathcal G.$
As $\C$ is a Stein space  the generalized Grauert theorem as proven in \cite{Bu} shows the existence of maps
$f_i\colon U_i\to\mathcal G$ satisfying $f_i f^{-1}_j=g_{i,j}.$
Then \[\nabla^\lambda_i.f_i=\nabla^\lambda_j.f_j\]
on $U_i\cap U_j$ and we obtain a well-defined $\C^*$-family of flat $\SL(2,\C)$-connections 
$\tilde \nabla^\lambda$
which satisfies the 
reality condition and the closing condition of Theorem \ref{general_reconstruction}. Applying the proof of Theorem
\ref{general_reconstruction} we see that the holomorphic structure $(\tilde \nabla^0)''$ is stable and therefore
the residuum of $\tilde \nabla^\lambda$ at $\lambda=0$ is a nowhere vanishing nilpotent complex linear 1-form.
By Theorem \ref{general_reconstruction} we obtain an immersed minimal surface $f\colon M\to S^3.$
The formula for the energy of $f$ can be computed by similar methods as used in Section \ref{exceptional_connections}.

Now assume that $p$ only branches at those $\mu\in\Sigma$ where 
$\mathcal L(\mu)=\underline\C\in Jac(\tilde M/ \Z_3).$ Then the map $\mathcal D$ into the moduli space
of flat $\SL(2,\C)$-connections can be locally lifted (denoted by $\nabla_i^\lambda$) to the space of flat connections in 
such a way that a corresponding nowhere vanishing family of Higgs fields $\Psi^i_\lambda$ has non-zero
determinant whenever $\mathcal L(\mu)\neq\underline\C\in Jac(\tilde M/\Z_3)$ for $p(\mu)=\lambda.$
Arguing in the same lines as in the proof of Lemma \ref{breaking_symmetries} one sees that all
families of connections $\nabla^\lambda_i$ are gauge equivalent to
$\varphi^*\nabla^\lambda_i$ by holomorphic families of gauges  for all symmetries $\varphi=\varphi_2,\,\varphi_3,\,\tau.$ Then the uniqueness part
of Theorem \ref{general_reconstruction} proves that the corresponding minimal surface has the symmetries
$\varphi_2,$ $\varphi_3$ and $\tau.$ Moreover, Theorem \ref{dressing} and Lemma \ref{breaking_symmetries}
show the uniqueness of this minimal surface.

\end{proof}
\begin{Rem}
Computer experiments in \cite{HS} suggest that the spectral curve of the Lawson surface of genus $2$ is
not branched over
the punctured unit disc $D=\{\lambda\in\C\mid 0<\parallel\lambda\parallel\leq1\}.$ With these numerical spectral 
data the Lawson surface of genus $2$  can be visualized as a conformal immersion from the Riemann surface $M$ into 
$S^3$ by an
implementation of Theorem \ref{reconstruction_spec}
 in the xlab software of Nicholas Schmitt (see Figure \ref{Lawson_figure}).
\end{Rem}

\begin{figure}[htbp]
  \centering
  \begin{minipage}[b]{5.5 cm}
    {\vbox{\hspace{-1.5cm}
\vbox{\vspace{0.5cm}
\includegraphics[scale=0.25]{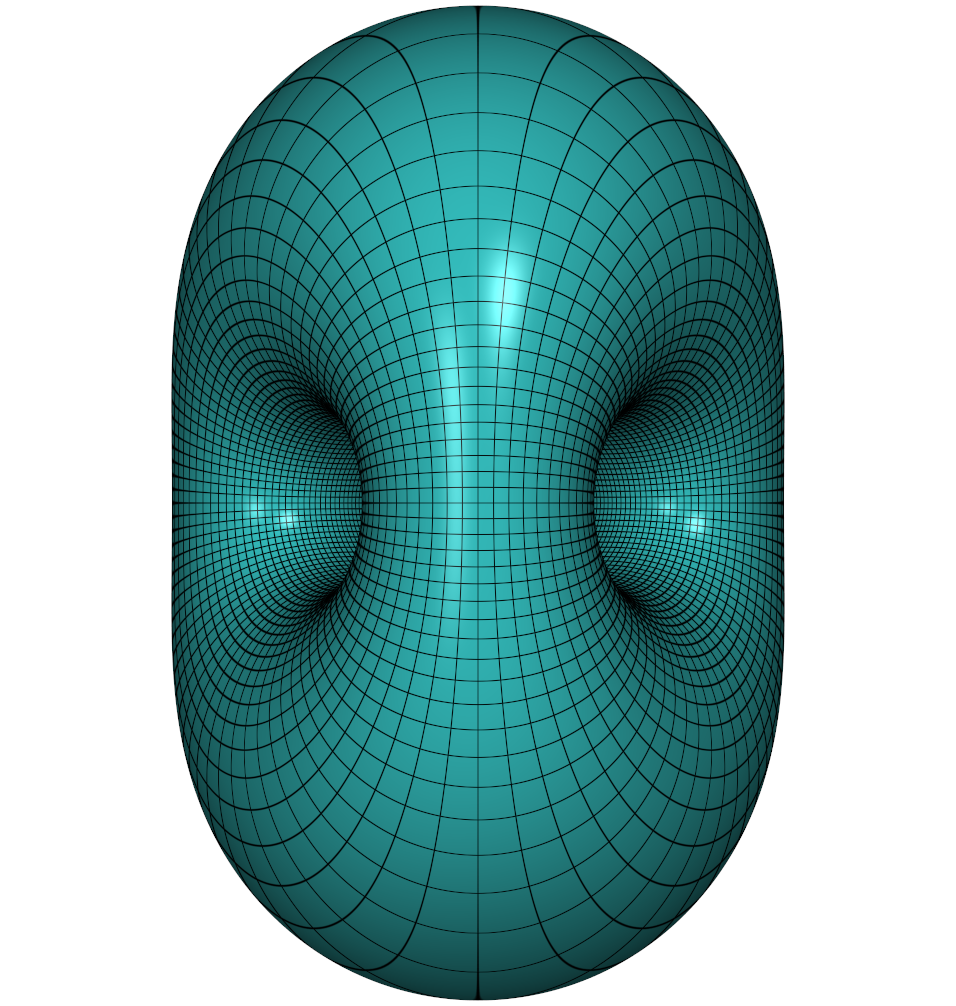}
}
}}  
  \end{minipage}
 \caption{Lawson genus 2 surface, picture by Nicholas Schmitt.}
  \label{Lawson_figure}
\end{figure}

\section{Lawson symmetric CMC surface of genus $2$}\label{A3}
In \cite{HS} we found numerical evidence that there exist a deformation of the Lawson surface of genus $2$ through
 compact CMC surface $f\colon M\to S^3$
 of genus $2$ which
preserves the extrinsic symmetries $\varphi_2,$ $\varphi_3$ and $\tau.$ 
We call these surfaces Lawson symmetric CMC
surfaces. We shortly explain how to generalize our theory to Lawson symmetric CMC surfaces.

Due to the Lawson correspondence, one can treat CMC surfaces in $S^3$ in the same way as minimal
surfaces, see for example \cite{B}. Consequently, there also exists an associated family of flat $\SL(2,\C)$-connections $\lambda\in\C^*\mapsto \nabla^\lambda$ 
which are unitary along the unit circle. 
In contrast to the minimal case
the Sym points $\lambda_1\neq\lambda_2\in S^1$,
at which the connections $\nabla^{\lambda_i}$ are trivial,
must not be the negative of each other.
 Then, the
CMC surface is obtained as the gauge between these two flat connections, but the mean curvature is now given by
$H=i\frac{\lambda_1+\lambda_2}{\lambda_1-\lambda_2}.$ For $\lambda_1=-\lambda_2$ we get a
 minimal surface. 
 
 As the extrinsic symmetries $\varphi_2,$ $\varphi_3$ and $\tau$ are (assumed to be) holomorphic on the surface,
the Riemann surface structure is almost fixed: It is given by the algebraic equation
\[y^3=\frac{z^2-a}{z^2+a}\]
for some $a\in \C^*.$ The Lawson Riemann surface structure is then given by $a=1.$
 Moreover, the every individual connection $\nabla^\lambda$ of the associated family is equivariant
with respect to the Lawson symmetries. All the theory developed for flat Lawson symmetric
$\SL(2,\C)$-connections on the Lawson surface carries over to flat Lawson symmetric $\SL(2,\C)$-connections
on $M:$
 The moduli space of Lawson symmetric holomorphic structures is double covered by the Jacobian of a complex 
$1$-dimensional torus. This torus itself is given by the equation $y^2=\frac{z^2-a}{z^2+a}.$
 There is a $2:1$ correspondence between gauge equivalence classes of flat line bundle connections on the above 
mentioned torus and gauge equivalence classes of flat 
Lawson symmetric $\SL(2,\C)$-connections on $M$ away from divisors in the corresponding moduli spaces.
 The correspondence extends to these divisors in the sense of Theorem \ref{extension_dbar0} and 
Theorem \ref{extension_non_stable}. The concrete formulas are analogous to the case of the Lawson surface.

From the observation that the moduli spaces of the flat Lawson symmetric $\SL(2,\C)$-connections can be described
analogously to the case of the Lawson surface itself, it is clear that the definition and the basic
properties of the spectral curve carries over to Lawson symmetric CMC surfaces of genus $2.$ Of course, the 
extrinsic closing condition changes, as well as the precise form of the energy formula.

\appendix

\section{The associated family of flat connections}\label{A1}
In this appendix we shortly recall the gauge theoretic description of minimal surfaces in $S^3$ which is
due to Hitchin \cite{H}. For more details, one can also consult \cite{He}. 

The Levi-Civita connection of the round $S^3$ is given with respect to the left trivialization $TS^3=S^3\times\im\H$ as
\[ \nabla=d+\frac{1}{2}\omega,\]
where $\omega$ is the Maurer-Cartan form of
$S^3$ which acts via adjoint representation.

The hermitian complex rank $2$ bundle $V=S^3\times\H$
with complex structure given by right multiplication with $i\in\H$ is a spin bundle for $S^3:$
The Clifford multiplication is given by
$TS^3\times V\to V; (\lambda, v)\mapsto \lambda v$
where $\lambda\in\Im \H$ and $v\in\H,$ 
and this identifies $TS^3$ as the skew symmetric 
trace-free complex linear endomorphisms of $V.$ 
There is an unique complex unitary connection on $V$ which induces on $TS^3\subset\End(V)$
 the Levi-Civita connection.
It is given by
\[
\nabla=\nabla^{spin}=d+\frac{1}{2}\omega,
\]
where the $\Im\H-$valued Maurer-Cartan form acts by left multiplication in the quaternions. 
 
Let $M$ be a Riemann surface and $f\colon M\to S^3$ be a conformal 
immersion. Then the pullback $\phi=f^*\omega$ of the Maurer-Cartan form satisfies the structural equations
\begin{equation}\label{MC}
d^\nabla\phi=0,
\end{equation}
where $\nabla=f^*\nabla=d+\frac{1}{2}\phi.$ 
The conformal map $f$ is minimal if and only if it is harmonic, i.e., if
\begin{equation}\label{harmonic}
d^\nabla*\phi=0.
\end{equation}
holds.
Let
\[\frac{1}{2}\phi=\Phi-\Phi^*
\]
be the decomposition of $\phi\in\Omega^1(M;f^*TS^3)\subset\Omega^1(M;\End_0(V))$ into the complex linear and 
complex anti-linear parts.
As $f$ is conformal 
\[\det\Phi=0.
\]
Note that $f$ is an immersion if and only if $\Phi$ is nowhere vanishing. 
In that case $\ker\Phi=S^*$ is the dual to the holomorphic spin bundle $S$ associated to the immersion.
The Equations
\ref{MC} and \ref{harmonic} are equivalent to
\begin{equation}\label{dbarPhi}
\nabla''\Phi=0,
\end{equation}
where $\nabla''=\frac{1}{2}(d^\nabla+i*d^\nabla)$ 
is the underlying holomorphic structure of the pull-back of the spin connection on $V.$
Of course \eqref{dbarPhi} does not contain
the property that $\nabla-\frac{1}{2}\phi=d$ is trivial. 
Locally this is equivalent to
\begin{equation}\label{curvature}
F^\nabla=[\Phi\wedge\Phi^*]
\end{equation}
as one easily computes.

From \eqref{dbarPhi} and \ref{curvature} one sees that
the associated family of connections
\begin{equation}\label{nablafamily}
\nabla^\lambda:=\nabla+\lambda^{-1}\Phi-\lambda\Phi^*
\end{equation}
is flat for all $\lambda\in\C^*,$ unitary along $S^1\subset\C^*$ and trivial for $\lambda=\pm1.$
 This family contains all the informations about the surface, i.e., given such a family of flat
connections one can reconstruct the surface as the gauge between $\nabla^1$ and $\nabla^{-1}.$  Using 
Sym-Bobenko formulas one can also make CMC surfaces in $S^3$ and $\R^3$ out of the family of flat connections. 
These CMC surfaces do not close in general. 

The family of flat connections can be written down in terms of the well-known geometric data associated to a minimal surface:
\begin{Pro}\label{connection_data}
Let $f\colon M\to S^3$ be a conformal minimal immersion with associated complex unitary rank $2$ bundle $(V,\nabla)$
and with induced spin bundle $S.$
 Let $V= S^{-1}\oplus S$ be the unitary decomposition, where $S^{-1}=\ker\Phi\subset V$ and $\Phi$ is the $K-$part of 
 the differential of $f.$ The Higgs field $\Phi\in H^0(M,K\End_0(V))$ can be identified with
 \[\Phi=\frac{1}{2}
 \in H^0(M;K\Hom(S,S^{-1})),\]
  and its adjoint $\Phi^*$ is given by $i\vol$ where $\vol$ is the volume form of the induced Riemannian metric.
The family of flat connections is given by
\[\nabla^\lambda=\dvector{&\nabla^{spin^*} & 
-\frac{i}{2} Q^*\\ & -\frac{i}{2} Q & \nabla^{spin}}+\lambda^{-1}\Phi-\lambda\Phi^*,\]
where $\nabla^{spin}$ is the spin connection corresponding to the Levi-Civita connection on $M$ and $Q$ is the Hopf 
differential of $f.$
\end{Pro}

\section{Lawson's genus $2$ surface}\label{A2}
We recall the construction of Lawson's minimal surfaces of genus $2$ in $S^3,$ see \cite{L}. 
Consider the round $3-$sphere
\[S^3=\{(z,w)\in\C^2\mid |z|^2+|w|^2=1\}\subset\C\oplus\C\]
and the geodesic circles
$C_1=S^3\cap ( \C\oplus\{0\})$ and $C_2=S^3\cap (\{0\}\oplus\C)$ on it.
Take the six points \[Q_k=(e^{i\frac{\pi}{3}(k-1)},0)\in C_1\] in equidistance on $C_1,$ and the four points 
\[P_k=(0,e^{i\frac{\pi}{2}(k-1)})\in C_2\] in equidistance on $C_2.$
A fundamental piece of the Lawson surface is the solution to the Plateau problem for the closed geodesic convex 
polygon $\Gamma=P_1Q_2P_2Q_1$ in $S^3.$  This means that it is the smooth minimal surface which is area 
minimizing under all surfaces with boundary $\Gamma.$
To obtain the Lawson surface one reflects the fundamental piece along the geodesic through $P_1$ and $Q_1,$ 
then one rotates everything around the geodesic $C_2$ by $\frac{2}{3}\pi$ two times, and in the end one reflects the 
resulting surface across the geodesic $C_1.$
Lawson has shown that the surface obtained in this way is smooth at all points. It is embedded, orientable and
has genus $2.$
The umbilics, i.e., the zeros
of the Hopf differential $Q$ are exactly at the points $P_1,..,P_4$ of order $1.$ 

 A generating system 
of the symmetry group of the Lawson surface is given by
\begin{itemize}
\item the $\Z^2-$action generated by $\varphi_2$ with $(a,b)\mapsto(a,-b);$ it is orientation preserving on the surface and its fix points are $Q_1,..Q_6;$
\item the $\Z_3-$action  generated by the rotation $\varphi_3$ around $P_1P_2$ by $\frac{2}{3}\pi,$ i.e.,
$(a,b)\mapsto(e^{i\frac{2}{3}\pi}a,b),$ 
which is holomorphic on $M$ with fix points $P_1,..,P_4;$
\item the reflection at $P_1Q_1,$ which is 
antiholomorphic; it is given by $\gamma_{P_1Q_1}(a,b)=(\bar a,\bar b);$
\item the reflection at the sphere $S_1$ corresponding to the 
real hyperplane spanned by $(0,1), (0,i),(e^{\frac{1}{6}\pi i},0),$ with $\gamma_{S_1}(a,b)=(e^{\frac{\pi}{3}i}\bar a, b);$
it is antiholomorphic on the surface,
\item the reflection at the sphere $S_2$ corresponding to the 
real hyperplane spanned by $(1,0), (i,0),(0,e^{\frac{1}{4}\pi i}),$ which is antiholomorphic on the surface and satisfies
$\gamma_{S_2}(a,b)=(a,i\bar b).$
\end{itemize}
Note that all these actions commute with the 
$\Z_2-$action. The last two fix the polygon 
$\Gamma.$ They and the first two map the oriented normal to itself. The third one maps the oriented normal to its negative.

Using the symmetries, one can determine the Riemann surface structure of 
the Lawson surface $f\colon M\to S^3$ as well as the other holomorphic data associated to it:
\begin{Pro}\label{LawsonRS}
The Riemann surface $M$ associated to the Lawson genus $2$ surface is the three-fold covering $\pi\colon M\to \CP^1$ 
of the Riemann sphere with branch points of order $2$ over $\pm1,\pm i\in\CP^1,$ i.e., the compactification of the
algebraic curve
\[y^3=\frac{z^2-1}{z^2+1}.\]
The hyper-elliptic involution is given by $(y,z)\mapsto(y,-z)$ and the Weierstrass points are $Q_1,..,Q_6.$
The Hopf differential of the Lawson genus $2$ surface is given by
\[Q=\pi^*\frac{ir}{z^4-1}(dz)^2\]
for a nonzero real constant $r\in\R$ and the spin bundle $S$ of the immersion is
\[S=L(Q_1+Q_3-Q_5).\]
\end{Pro}

\end{document}